\documentclass[12pt]{article}

\usepackage{fullpage}
\usepackage[T1]{fontenc}
\usepackage[utf8]{inputenc}
\usepackage{lmodern}
\usepackage{microtype}
\usepackage{amsmath,amssymb,amsfonts,amsthm,mathtools}
\usepackage{enumitem}
\usepackage[hidelinks]{hyperref}
\usepackage{cite}

\newcommand{\Bb}{\mathbb{B}}
\newcommand{\Zz}{\mathbb{Z}}
\newcommand{\Rr}{\mathbb{R}}
\newcommand{\Id}{\operatorname{Id}}
\newcommand{\DUC}{\mathfrak{D}}

\newcommand{\norm}[1]{\left\lVert #1\right\rVert}
\newcommand{\abs}[1]{\left|#1\right|}
\DeclareMathOperator{\Lip}{Lip}
\DeclareMathOperator{\dist}{dist}

\newtheorem{thm}{Theorem}[section]
\newtheorem{lem}[thm]{Lemma}

\newtheorem{corr}[thm]{Corollary}
\theoremstyle{definition}
\newtheorem{defin}[thm]{Definition}
\newtheorem{example}[thm]{Example}
\newtheorem{remark}[thm]{Remark}

\title{Discontinuous Generalized Hyperbolicity and Structural Stability
}
\author{Sergey Tikhomirov\thanks{Institute of Mathematics and Statistics,
University of S\~ao Paulo (IME--USP), S\~ao Paulo, Brazil.
Corresponding author: sergey.tikhomirov@gmail.com.
ORCID: 0000-0002-0018-2356.}}

\date{}

\begin{document}
\maketitle

\begin{abstract}
We study diffeomorphisms of Banach spaces satisfying generalized
hyperbolicity: the stable and unstable subspaces are invariant only up
to inclusion, need not depend continuously on the point, and
expansivity may fail. On closed manifolds, this condition is equivalent
to Axiom~A with strong transversality. We introduce a
continuous-selection method that
converts uniform orbitwise estimates into bounded continuous solutions
of cohomological equations for generalized-hyperbolic cocycles. For
diffeomorphisms in the uniform global \(C^1\) category, this yields
continuous semiconjugacies in both directions under sufficiently small
bounded Lipschitz perturbations, and topological stability in the sense
of Walters. For the infinite product of one-dimensional Morse-Smale
systems on \(\ell^p(\mathbb Z)\), \(1\leq p\leq\infty\), we construct
compatible forward and reverse cohomological solution operators and
prove structural stability, including for perturbations that couple
coordinates, although the product is not expansive. For
\(p=\infty\) this gives a structurally stable diffeomorphism with
uncountably many hyperbolic fixed points, in contrast with the finitely
many basic sets of Axiom~A diffeomorphisms. The cocycle analysis 
yields Lipschitz shadowing for the induced operator on bounded
continuous vector fields.


\end{abstract}

\section{Introduction}

Hyperbolicity, introduced by Anosov \cite{Anosov} on the whole manifold
and by Smale \cite{Smale} on the nonwandering set (Axiom~A), provides
the classical framework for shadowing and structural stability. For
$C^1$-diffeomorphisms of smooth closed manifolds, structural stability
is characterized by Axiom~A and the strong transversality condition
\cite{WenBook}. We develop a stability theory for global dynamics on
Banach spaces based on generalized hyperbolicity \cite{TikhomirovGH},
allowing discontinuous splittings and the absence of expansivity.

We work with $C^1$-diffeomorphisms of a Banach space $\Bb$ whose
derivatives and inverse derivatives are globally bounded and whose
derivatives are uniformly continuous. This is the class $\DUC(\Bb)$
of Section~\ref{sec:setting}. Generalized hyperbolicity requires a
splitting $\Bb=E_x^s\oplus E_x^u$ with uniformly bounded projections
and uniform forward and backward exponential contraction on the
respective spaces. The invariance conditions are inclusions, which
may be strict, and no continuity of the splitting or
finite-dimensionality of the unstable spaces is assumed. The
definition builds on generalized hyperbolicity for linear operators
\cite{BernardesEtal,CiriloGollobitPujals} and Pilyugin's
$(C,\lambda)$-structure \cite{PilyuginGeneralizations}.

On smooth closed manifolds, the tangent-bundle version of this
definition is equivalent to Axiom~A and strong transversality
\cite[Theorem~3]{TikhomirovGH}. Allowing discontinuity is necessary
even for elementary Morse-Smale systems: for a north-south
diffeomorphism of a circle, the stable dimension is zero at the source
and full at the sink, so a global generalized-hyperbolic splitting
cannot be continuous.

Generalized hyperbolicity allows dynamics that classical hyperbolicity
excludes. Already for linear operators, generalized-hyperbolic maps can
have shadowing without expansivity and a dense set of points homoclinic
to the origin \cite{BernardesEtal,CiriloGollobitPujals}. Nonlinear
shifts provide a nonlinear counterpart \cite[Example~19]{TikhomirovGH},
and the Morse--Smale product of \cite[Example~20]{TikhomirovGH} exhibits
arbitrarily large variation in stable and unstable dimensions within a
single system. 

We introduce a new method for stability based on continuous selection.
Orbitwise estimates define convex sets of admissible initial vectors for
finite-time recurrence problems. Since the recurrence data are continuous
and these sets contain strictly feasible points, they admit continuous
selections, by Michael's theorem \cite{MichaelI} or directly by a
partition of unity (Remark~\ref{rem:partition-of-unity}). Averaging and
successive corrections then yield exact solutions of cohomological
equations, even when the underlying splitting is discontinuous.

This method yields our general stability theorem: a
generalized-hyperbolic diffeomorphism admits continuous
semiconjugacies in both directions under sufficiently small bounded
Lipschitz perturbations. More precisely, if $f\in\DUC(\Bb)$ satisfies
generalized hyperbolicity and $g=f+\Delta$, where
$\norm\Delta_\infty,\Lip(\Delta)\leq d$ with $d>0$ sufficiently small,
then there exist bounded continuous fields $h,k$ such that
\[
 g\circ(\Id+h)=(\Id+h)\circ f,
 \qquad
 f\circ(\Id+k)=(\Id+k)\circ g,
\]
with $\norm h_\infty,\norm k_\infty\leq Md$, where $M$ is independent
of the perturbation (Theorem~\ref{thm:main}).
The second semiconjugacy requires much less of the perturbation:
$\Id+k$ exists, with the same bound, for every homeomorphism $g$ of
$\Bb$ with $\sup_{x\in\Bb}\abs{g(x)-f(x)}\leq d$, without any Lipschitz
assumption on $g-f$ (Theorem~\ref{thm:topological-stability}).
Thus generalized hyperbolicity implies topological stability in the
sense of Walters \cite{WaltersTopStability}
(Definition~\ref{def:stability}).  Remark~\ref{rem:trajectory-set-stability}
compares this with the classical results and records the resulting
estimate between complete-trajectory sets in the uniform-in-time
Hausdorff distance, a quantitative form of strong tolerance stability
\cite{MazurTolerance}.

Passing from semiconjugacy to full structural stability raises a
further difficulty: without expansivity, the usual uniqueness
argument does not show that the two semiconjugacies are mutually
inverse. Continuous selection provides solutions of the cohomological
equations but no canonical choice of them, so the composition of the
two semiconjugacies is not controlled. We overcome this difficulty for
the infinite Morse-Smale product $f(x)_j=a(x_j)$ from
\cite[Example~20]{TikhomirovGH}, under the uniform scalar assumptions of
Subsection~\ref{sec:product-results}. Every generalized-hyperbolic
splitting of this product is discontinuous, so the earlier stability
theorem \cite[Theorem~16]{TikhomirovGH} does not apply. Our structural stability theorem holds on every $\ell^p(\Zz)$, $1\leq p\leq+\infty$, and allows
perturbations coupling all coordinates
(Theorem~\ref{thm:product-structural-stability}).

For this product, both semiconjugacy equations are solved with Green
operators compatible with composition. The forward equation is the same
as in Theorem~\ref{thm:main}, and its solution operator is linear. The
reverse equation, in contrast, is a nonlinear cohomological equation:
its coefficient $Df(y+u(y))$ is evaluated at the point corrected by the
unknown displacement $u$, not at $y$ as in Theorem~\ref{thm:main}.
Accordingly, its solution operator is nonlinear and is defined by a
normalization in which the cutoff is also evaluated at the corrected
point. Pullback and remainder estimates then force the semiconjugacies
to be mutually inverse without expansivity.

In particular, for $p=+\infty$ the product is a structurally stable
diffeomorphism with uncountably many hyperbolic fixed points, pairwise
at distance at least $1$.
A description in the spirit of
Axiom~A and strong transversality, based on a decomposition of the
nonwandering set into finitely many basic sets ordered by a filtration,
does not seem available for such a system. Generalized hyperbolicity is
imposed on the whole phase space and requires no such decomposition.
For $p<+\infty$ the fixed points form a
countably infinite set with finite but unbounded stable dimensions. None
of these products has a compact global attractor
(Corollary~\ref{rem:product-uncountably-many}).

The continuous-selection argument underlying the general
semiconjugacy theorem applies more broadly to cocycles. For a
generalized-hyperbolic cocycle $(\alpha,A)$ satisfying the uniform
bounds and uniform continuity assumptions of
Subsection~\ref{sec:induced-shadowing-results}, every
bounded continuous forcing $w$ admits a bounded continuous solution of
\[
 v(\alpha(x))-A(x)v(x)=w(x),
 \qquad
 \norm v_\infty\leq K_c\norm w_\infty,
\]
where $K_c$ depends only on the hyperbolicity constants
(Theorem~\ref{thm:cocycle-D-surj}). For derivative cocycles, this
gives a global Banach-space analogue of infinitesimal stability
\cite{ManeInfinitesimal,ManeAS}. The semiconjugacy proof applies
this result to $(f,Df)$ and $(g,Df)$, after establishing that
generalized hyperbolicity persists under the change of base from
$f$ to $g$.
The cocycle analysis also gives Lipschitz shadowing for the induced
operator $T_{\alpha,A}$ on $C_b^0(\Bb,\Bb)$, defined by
$(T_{\alpha,A}v)(\alpha(x))=A(x)v(x)$
(Theorem~\ref{thm:induced-shadowing}).

The earlier work \cite[Theorem~11]{TikhomirovGH} established
Lipschitz shadowing, Lipschitz periodic shadowing, density of periodic
points in the chain-recurrent set, and robustness in the uniform
$C^1$ distance. Its stability theorem required a uniformly continuous
splitting and, for full conjugacy, an additional compatibility
condition \cite[Theorem~16]{TikhomirovGH}. The present results extend
the semiconjugacy theory to discontinuous splittings and establish
structural stability of the product beyond those earlier hypotheses.
Together, these results support generalized hyperbolicity as an
analogue of Axiom~A and strong transversality for global Banach-space
dynamics. Whether generalized hyperbolicity always implies
structural stability in this category remains open.

Section~\ref{sec:setting} gives the setting, states the results,
and presents the proof scheme and the open questions; the remaining
sections contain the proofs.
Section~\ref{sec:Michael} collects the dynamical and
functional-analytic tools. Section~\ref{sec:linear-input} develops
the cocycle theory, and Section~\ref{sec:semiconjugacies} constructs
the semiconjugacies. Section~\ref{sec:infinite-morse-smale} proves
structural stability of the product, with the scalar estimates
established in Appendix~\ref{app:product-scalar-pseudotrajectories}.

\subsection{Context and related directions}
\label{sec:historical-overview}

The notion of structural stability goes back to Andronov and Pontryagin's
work on planar differential equations \cite{AndronovPontryagin}.
Peixoto subsequently characterized structurally stable vector fields
on compact surfaces \cite{PeixotoSurfaces}.

\textbf{Classical stability, shadowing, and continuous solvability.}
For diffeomorphisms of smooth closed manifolds, the Palis--Smale program
\cite{PalisSmale} led to the characterization of $C^1$ structural
stability by Axiom~A and strong transversality: sufficiency was proved
by Robbin and Robinson \cite{Robbin,Robinson}, and necessity by
Ma\~n\'e \cite{ManeStability}, building on his ergodic closing lemma
\cite{ManeClosing}.

The shadowing theory originating with Anosov and Bowen
\cite{Anosov,BowenShadowing} gives a quantitative counterpart.
Pilyugin and the author proved that Lipschitz shadowing characterizes
structural stability \cite{PilyuginTikhomirov}; see also
\cite{PilyuginSakai,TikhomirovHolder}. 
Variational shadowing gives another characterization \cite{PilyuginVariational}. 
Related developments include the characterization of $\Omega$-stability by Lipschitz periodic
shadowing \cite{OsipovPilyuginTikhomirov}, and central shadowing for
partially hyperbolic systems \cite{KryzhevichTikhomirov}.

The analytical approach studies the conjugacy equation on spaces of
continuous maps \cite{MoserAnosov}. Mather
\cite{Mather} characterized Anosov diffeomorphisms by hyperbolicity
of the induced operator on continuous vector fields, while Ma\~n\'e
related surjectivity of the linearized conjugacy operator to Axiom~A
and strong transversality \cite{ManeInfinitesimal,ManeAS}.

\textbf{Endomorphisms and the choice of phase space.}
For noninvertible maps, unstable directions may depend on the
prehistory, leading naturally to inverse-limit spaces
\cite{PrzytyckiAnosov}.  Closing lemmas
provide foundations for their stability theory
\cite{WenClosingNonsingular,WenClosingSingularities};
see \cite{AnosovZhuzhomaClosing} for a survey of closing lemmas.
Berger and Kocsard \cite{BergerKocsard} prove $C^1$ inverse-limit
structural stability under Axiom~A and their strong transversality
condition, including for maps with critical points.  This conjugacy
need not descend to the original manifold, where additional
conditions are required
\cite{BergerAttractorRepellor,IglesiasEtalTorus,ManePughEndomorphisms,PrzytyckiAnosov}.
Extending our continuous-solvability
approach to the original phase space would require compatibility
between different prehistories.

\textbf{Linear dynamics and cocycles in Banach spaces.}
Pugh's Banach-space version of the Grobman--Hartman theorem
\cite{PughHartman} establishes structural stability of hyperbolic
linear operators under small bounded Lipschitz perturbations.
Generalized hyperbolicity \cite{BernardesEtal,CiriloGollobitPujals, BernardesLocallyConvex}
allows strict invariance inclusions and admits nonhyperbolic
operators with shadowing and structural stability
\cite{BernardesMessaoudi}. These results motivate the nonlinear
definition of \cite{TikhomirovGH} and the search for conjugacy
arguments without expansivity.

The connection between bounded solvability, exponential dichotomies,
shadowing, and stability is developed in
\cite{DragicevicAdmissibility,HenryDichotomies,PalmerDichotomiesSurvey}.
Our cocycle arguments retain this use of orbitwise estimates and add
a continuous-selection construction to obtain solutions continuous
in the base point. We also use the passage from finite-time to two-sided
shadowing for linear operators due to Bernardes and Peris
\cite{BernardesPeris}. A different use of cocycles appears in
probabilistic shadowing for linear skew products
\cite{MonakovTikhomirov,TikhomirovSkewProducts}.

\textbf{Flows, singularities, and variation of stable dimensions.}
For flows, structural stability is understood up to orbit equivalence
\cite{HayashiFlows,RobinsonFlows,WenFlows}, and Lipschitz shadowing
allows controlled changes of time
\cite{PalmerPilyuginTikhomirov,PilyuginFlows}; see
\cite{FisherHasselblatt}.  The choice of time reparametrization matters:
$C^1$-robust oriented shadowing implies $\Omega$-stability
\cite{GanLiTikhomirov}, but need not imply structural stability
\cite{PilyuginTikhomirovOriented}.

To treat singularities, the extended linear Poincar\'e flow retains
limiting flow directions \cite{LiGanWenPoincare}.  Combined with
reparametrization cocycles, it leads to multisingular hyperbolicity,
which accommodates singularities of different stable indices
\cite{BonattiDaLuz}.  This connects conceptually with the dimension
variation motivating our definition.  In that framework the splittings
are continuous over an enlarged base; ours may be discontinuous over
the original phase space.

\textbf{Delay and evolution equations.}
Noninvertible dynamics on function spaces also arise from delay and
evolution equations
\cite{DiekmannDelay,HaleVerduynLunel,HenryParabolic,SellYou}.
The hyperbolicity theories of Lani-Wayda and Steinlein--Walther
give shadowing and symbolic dynamics for noninvertible Banach-space
maps \cite{LaniWayda,SteinleinWaltherShadowing,SteinleinWalther}.
They are connected with hyperbolic periodic orbits and homoclinic
dynamics in delay equations
\cite{LaniWaydaWaltherChaosI,WaltherHomoclinic}, and with persistence
of Poincar\'e dynamics under perturbations
\cite{LaniWaydaPersistence}. These are the closest earlier notions of
hyperbolicity for nonlinear Banach-space dynamics. Our setting permits
discontinuous projections and strict inclusions in both invariance
conditions \cite{TikhomirovGH}; expansivity may consequently fail,
whereas it is available in Lani-Wayda's shadowing and persistence
theory. 

For evolution equations, Arrieta, Carvalho, and Takaessu
\cite{ArrietaCarvalhoTakaessu} prove shadowing near the compact
global attractor of a Morse-Smale semigroup in a Hilbert space.
Their assumptions allow weaker global derivative regularity than
ours, but require finitely many equilibria and finite-dimensional
unstable spaces. Our product theorem concerns global dynamics with
infinitely many hyperbolic fixed points and no compact global
attractor. Extending the present methods to semiflows would require
changes to the treatment of backward trajectories and to the
regularity assumptions; see Subsection~\ref{sec:future-research}.

\textbf{Fluid dynamics.}
The work of Mailybaev and Raibekas on spontaneous stochasticity
\cite{MailRaib1,MailRaib2} suggests extensions to Fr\'echet spaces or
manifolds. Gravitational fingering in incompressible porous media
\cite{otto-menon-2005} leads, through travelling-wave reductions of
related models \cite{PetrovaTikhomirovEfendiev2025}, to questions of
normal hyperbolicity for flows in Banach spaces. These applications were the original motivation for this work, but remain beyond the present framework; our results are a step towards
such extensions.

The questions raised by these directions are collected in
Subsection~\ref{sec:future-research}.

\section{Setting and statements of the results}
\label{sec:setting}

\subsection{The uniform global \texorpdfstring{$C^1$}{C1} category}

Let $\Bb$ be a Banach space.  We use the uniform global $C^1$ category \cite{TikhomirovGH}.  
Let $f:\Bb\to\Bb$ be a $C^1$-diffeomorphism and suppose that there exist $R>0$ and a nondecreasing function $r:\Rr^+\to\Rr^+$ such that
\begin{equation*}
 \abs{f(x+v)-f(x)-Df(x)v}\leq \abs v\,r(\abs v),
 \qquad x,v\in\Bb,
\end{equation*}
\begin{equation}
\label{eq:C1-2}
 \norm{Df(x)}\leq R,
 \qquad
 \norm{D(f^{-1})(x)}\leq R,
 \qquad x\in\Bb,
\end{equation}
\begin{equation}
\label{eq:C1-r1}
 \norm{Df(x+v)-Df(x)}\leq r(\abs v),
 \qquad x,v\in\Bb,
\end{equation}
\begin{equation}
\label{eq:C1-r2}
 \lim_{t\to0}r(t)=0.
\end{equation}
We denote the class of such diffeomorphisms by $\DUC(\Bb)$.  For
$f,g\in\DUC(\Bb)$, the extended uniform $C^1$ distance is
\begin{equation*}
 \norm{f-g}_{C^1}
 :=\sup_{x\in\Bb}\abs{f(x)-g(x)}
   +\sup_{x\in\Bb}\norm{Df(x)-Dg(x)}.
\end{equation*}
Its value is allowed to be infinite.

Define the nonlinear remainder
\begin{equation}
\label{eq:sdef}
 s_f(x,v):=f(x+v)-f(x)-Df(x)v,
 \qquad x,v\in\Bb.
\end{equation}
Note that for every $\rho>0$ and every $x,v,\widetilde v\in\Bb$ satisfying
$\abs v,\abs{\widetilde v}\leq\rho$, one has
\begin{equation}
\label{eq:s-rho}
 \abs{s_f(x,v)-s_f(x,\widetilde v)}
 \leq r(\rho)\abs{v-\widetilde v}.
\end{equation}
In particular, for every $\varepsilon>0$ there exists $\rho>0$ such that
\begin{equation*}
 \abs{s_f(x,v)-s_f(x,\widetilde v)}
 \leq\varepsilon\abs{v-\widetilde v}, \quad x\in\Bb, \abs v,\abs{\widetilde v}\leq\rho.
\end{equation*}

\subsection{Generalized hyperbolicity}

We use the same definition as in \cite{TikhomirovGH}.

\begin{defin}[Generalized hyperbolicity]
\label{def:GH}
A diffeomorphism $f\in\DUC(\Bb)$ satisfies \emph{generalized
hyperbolicity} if there exist constants $C>0$, $\lambda\in(0,1)$ and a
family of projections
$P_x,Q_x:\Bb\to\Bb$, $x\in\Bb$,
such that, writing
$E_x^s:=P_x(\Bb)$,
$E_x^u:=Q_x(\Bb)$,
the following conditions hold.
\begin{enumerate}[label=\textup{(GH\arabic*)},leftmargin=12mm]
\item \textbf{Splitting.} For every $x\in\Bb$,
\[
E_x^s\oplus E_x^u=\Bb, \qquad
 P_x+Q_x=\Id,
 \qquad
 \norm{P_x},\norm{Q_x}\leq C.
\]

\item \textbf{Inclusion.} For every $x\in\Bb$,
\[
 Df(x)E_x^s\subseteq E_{f(x)}^s,
 \qquad
 D(f^{-1})(x)E_x^u\subseteq E_{f^{-1}(x)}^u.
\]

\item \textbf{Exponential estimates.} For every $x\in\Bb$ and every
integer $n>0$,
\[
 \abs{Df^n(x)v_x^s}\leq C\lambda^n\abs{v_x^s},
 \qquad v_x^s\in E_x^s,
\]
\[
 \abs{D(f^{-n})(x)v_x^u}\leq C\lambda^n\abs{v_x^u},
 \qquad v_x^u\in E_x^u.
\]
\end{enumerate}
No finite-dimensionality assumption is imposed on $E_x^u$.
No continuity, uniform continuity, or measurability of the maps
$x\mapsto P_x,Q_x$ is assumed.
\end{defin}

The inclusions in \textup{(GH2)} may be strict, and the projections
need not depend continuously on the base point.  Pointwise application of
$P_x$ and $Q_x$ therefore need not preserve the continuity of vector fields.
Continuous cohomological solvability is established in
Theorem~\ref{thm:cocycle-D-surj}; the change
of base needed for the reverse equation is treated in
Theorem~\ref{thm:g-Df-GH}.
For compact hyperbolic sets, invariant splittings are continuous;
their finer regularity is studied in
\cite{KatokHasselblatt,HasselblattRegularity,HasselblattRegularityII,HasselblattWilkinson}.

\textbf{Compact-manifold examples.}
On smooth closed manifolds, generalized hyperbolicity is equivalent to
Axiom~A and strong transversality \cite[Theorem~3]{TikhomirovGH}; in
particular, Morse--Smale diffeomorphisms are examples
\cite{ShubGlobalStability}.  The simplest instance is a north--south
diffeomorphism of the circle, where the stable dimension is zero at the
source and full at the sink, so no generalized-hyperbolic splitting on
the whole circle is continuous.  The scalar map $a$ of
Subsection~\ref{sec:product-results}, with a source at $0$ and sinks at
$\pm1$, is the noncompact version of this picture, and the product built
from it inherits the discontinuity
(Corollary~\ref{rem:product-uncountably-many}).

\textbf{Weighted and nonlinear shifts.}
On $\ell^p(\Zz)$, $1\leq p\leq+\infty$, consider the weighted shift \cite{CiriloGollobitPujals, BernardesEtal}
\[
 (f(x))_{j+1}=w_jx_j,
 \qquad
 w_j=\begin{cases}2,&j<0,\\[1mm]1/2,&j\geq0.\end{cases}
\]
The stable and unstable spaces consist of sequences supported on the
nonnegative and negative indices, respectively.  Both invariant
inclusions are strict: shifting carries one unstable coordinate into
the stable indices.  Nonlinear versions have
$(f(x))_{j+1}=a_j(x_j)$, with scalar diffeomorphisms $a_j$ satisfying
the uniform expansion, contraction, and regularity assumptions of
\cite[Example~19]{TikhomirovGH}.
These shifts have constant splittings, yet are not expansive and admit
nonunique shadowing while remaining structurally stable
\cite{BernardesMessaoudi,TikhomirovGH}.  They illustrate phenomena
enabled by strict invariance inclusions; the product in
Subsection~\ref{sec:product-results} also exhibits discontinuity of the
splitting.

\subsection{Bi-sided semiconjugacy}

\begin{defin}[Semi-structural, structural, and topological stability]
\label{def:stability}
Let $f\in\DUC(\Bb)$.  We call $f$ \emph{semi-structurally stable} if
for every $\varepsilon>0$ there exists $d_0>0$ such that every
$g\in\DUC(\Bb)$ with $\norm{f-g}_{C^1}<d_0$ admits continuous maps
$h,k:\Bb\to\Bb$ satisfying
$
 \max\{\norm h_\infty,\norm k_\infty\}<\varepsilon
$
and, writing $H=\Id+h$ and $K=\Id+k$,
\[
 g\circ H=H\circ f,
 \qquad
 f\circ K=K\circ g.
\]
Here $H$ and $K$ are called semiconjugacies.\footnote{Surjectivity of either
map is not required: in finite dimensions, $\Id+h$ with $h$ bounded
and continuous is surjective by Brouwer's fixed point theorem, but in infinite-dimensional Banach spaces it
need not be.}

We call $f$ \emph{structurally stable} if, under the same quantifiers,
the maps can be chosen so that
$
 K\circ H=H\circ K=\Id.
$
Thus $H$ and $K$ are mutually inverse homeomorphisms.

We call $f$ \emph{topologically stable} if for every $\varepsilon>0$
there exists $d_0>0$ such that every homeomorphism $g:\Bb\to\Bb$ with
$\sup_{x\in\Bb}\abs{f(x)-g(x)}<d_0$ admits a continuous map
$k:\Bb\to\Bb$ with $\norm k_\infty<\varepsilon$ such that $K=\Id+k$
satisfies
$
 f\circ K=K\circ g.
$
\end{defin}

Semi-structural stability requires
semiconjugacies in both directions, but only for perturbations in
$\DUC(\Bb)$ that are small in the $C^1$ sense; topological stability
requires one semiconjugacy, but for every homeomorphism uniformly close
to $f$.  
Theorems~\ref{thm:main} and~\ref{thm:topological-stability} below show
that generalized hyperbolicity implies semi-structural stability and
topological stability.

\begin{thm}[Bi-sided semiconjugacy]
\label{thm:main}
Assume that $f\in\DUC(\Bb)$ satisfies generalized hyperbolicity.  Then
there exist constants $M>0$ and $d_0>0$ such that the following holds.
Let $\Delta:\Bb\to\Bb$ be a bounded Lipschitz map satisfying
 $\norm{\Delta}_{\infty}\leq d$,
 $\Lip(\Delta)\leq d$, with 
 $0<d<d_0$,
and put
$
 g:=f+\Delta$.
Then $g$ is a bi-Lipschitz homeomorphism, and there exist
$h,k\in C_b^0(\Bb,\Bb)$ satisfying

\begin{equation}
\label{eq:main-semiconjugacy-reverse}
  g\circ(\Id+h)=(\Id+h)\circ f, \qquad f\circ(\Id+k)=(\Id+k)\circ g,
\end{equation}
\begin{equation*}
 \abs{h(x)},\abs{k(x)}\leq Md,
 \qquad x\in\Bb.
\end{equation*}
\end{thm}

\begin{corr}
\label{cor:two-semiconjugacies}
Every $f\in\DUC(\Bb)$ satisfying generalized hyperbolicity is
semi-structurally stable.
\end{corr}

\begin{proof}
Given $\varepsilon>0$, take $M,d_0$ from Theorem~\ref{thm:main} and
choose $0<d<\min\{d_0,\varepsilon/M\}$.  If $g\in\DUC(\Bb)$ satisfies
$\norm{f-g}_{C^1}<d$, then $\Delta:=g-f$ has
$\norm\Delta_\infty<d$ and 
$
 \Lip(\Delta)
 <d.
$
Theorem~\ref{thm:main} gives the required semiconjugacies with
$\norm{h}_\infty,\norm{k}_\infty\leq Md<\varepsilon$.
\end{proof}

The reverse semiconjugacy requires much less of the perturbation: it
suffices that $g$ be a homeomorphism uniformly close to $f$.

\begin{thm}[Topological stability]
\label{thm:topological-stability}
Assume that $f\in\DUC(\Bb)$ satisfies generalized hyperbolicity.  Then
there exist constants $M>0$ and $d_0>0$ such that the following holds.
Let $g:\Bb\to\Bb$ be a homeomorphism satisfying
$
 \sup_{x\in\Bb}\abs{g(x)-f(x)}\leq d < d_0.
$
Then there exists $k\in C_b^0(\Bb,\Bb)$ satisfying
\begin{equation}
\label{eq:top-stability-semiconjugacy}
 f\circ(\Id+k)=(\Id+k)\circ g,
 \qquad
 \abs{k(x)}\leq Md,
 \quad x\in\Bb.
\end{equation}
\end{thm}

\begin{corr}
\label{cor:topological-stability}
Every $f\in\DUC(\Bb)$ satisfying generalized hyperbolicity is
topologically stable.
\end{corr}

\begin{proof}
Given $\varepsilon>0$, take $M,d_0$ from
Theorem~\ref{thm:topological-stability} and choose
$0<d<\min\{d_0,\varepsilon/M\}$.  Every homeomorphism $g$ with
$\sup_{x\in\Bb}\abs{f(x)-g(x)}<d$ then admits $k$ with
$f\circ(\Id+k)=(\Id+k)\circ g$ and $\norm k_\infty\leq Md<\varepsilon$.
\end{proof}

Both theorems are proved in Section~\ref{sec:semiconjugacies}:
Theorem~\ref{thm:topological-stability} is proved first, and the reverse
semiconjugacy in Theorem~\ref{thm:main} is deduced from it.

\begin{remark}[Topological and tolerance stability]
\label{rem:trajectory-set-stability}
Walters proved that Anosov diffeomorphisms are topologically stable
\cite{WaltersTopStability} and, more generally, that every expansive
homeomorphism of a compact metric space with the shadowing property is
topologically stable \cite{WaltersPOTP}.  Nitecki \cite{Nitecki} proved
lower semi-stability, a closely related one-sided notion, for
diffeomorphisms of smooth closed manifolds satisfying Axiom~A and the
strong transversality condition; these need not be expansive.
Theorem~\ref{thm:topological-stability} differs from both results in
its setting and in its proof.  The phase space is a Banach space, so
compactness is unavailable; the map need not be expansive, so the
uniqueness of shadowing trajectories, which yields the continuity of
the semiconjugacy in Walters' argument, is unavailable as well; and the
finite decomposition of the nonwandering set used in the theory of
Axiom~A diffeomorphisms does not seem available, as the product of
Subsection~\ref{sec:product-results} shows.  

The semiconjugacies also control complete trajectories.
For a map $q:\Bb\to\Bb$, let
\[
 \mathcal T_q:=
 \bigl\{\mathbf x=(x_n)_{n\in\Zz}\in\Bb^{\Zz}:
             x_{n+1}=q(x_n)\text{ for all }n\in\Zz\bigr\}
\]
be its complete-trajectory set.  Equip $\Bb^{\Zz}$ with the extended
metric
\[
 d_\infty(\mathbf x,\mathbf y):=
       \sup_{n\in\Zz}\abs{x_n-y_n}\in[0,+\infty].
\]
For nonempty $A,B\subset\Bb^{\Zz}$, consider the Hausdorff distance
\[
\dist(A,B):=\sup_{\mathbf x\in A}\inf_{\mathbf y\in B}
                        d_\infty(\mathbf x,\mathbf y),
 \qquad
 d_{H}(A,B):=\max\{\dist(A,B),\dist(B,A)\}.
\]
Under the assumptions of Theorem~\ref{thm:main},
$
 d_{H}(\mathcal T_f,\mathcal T_g)\leq Md.
$
Indeed, writing $H=\Id+h$ and $K=\Id+k$, the semiconjugacy identities imply that
$(H(x_n))_{n\in\Zz}\in\mathcal T_g$ for every
$\mathbf x\in\mathcal T_f$, and
$(K(y_n))_{n\in\Zz}\in\mathcal T_f$ for every
$\mathbf y\in\mathcal T_g$.  Each image trajectory is at
$d_\infty$-distance at most $Md$ from the original one, proving both
directed estimates.  Under the assumptions of
Theorem~\ref{thm:topological-stability} alone, the semiconjugacy $K$
still gives the one-sided estimate
$\dist(\mathcal T_g,\mathcal T_f)\leq Md$ for every homeomorphism $g$
uniformly close to $f$.

Thus $f$ has a quantitative form of \emph{strong tolerance stability},
relative to the bounded Lipschitz perturbations in
Theorem~\ref{thm:main}; compare \cite{MazurTolerance}.
The comparison is uniform over all time
indices; ordinary tolerance stability compares orbit closures as
subsets of phase space.  In shadowing terminology, the estimate from
$\mathcal T_f$ to $\mathcal T_g$ is inverse shadowing for methods
generated by these perturbations
\cite{CorlessPilyuginInverse,PilyuginInverseContinuous,PilyuginVolfsonTodorov}.
For the related theory of
bishadowing, including applications to delay equations, see
\cite{KloedenOmbachBishadowing,AlNayefBishadowing}.
For dissipative PDEs, related continuity questions are studied through
trajectory attractors in spaces of solutions, typically with
local-in-time topologies and without requiring uniqueness
\cite{VishikChepyzhovTrajectories, ChepyzhovVishikAttractors}.
\end{remark}

\subsection{Infinite product of Morse--Smale systems}
\label{sec:product-results}

Our structural stability theorem concerns the following product model. We revisit the coordinatewise product of
\cite[Example~20]{TikhomirovGH} and establish its structural
stability.  We state the scalar assumptions in the form used below.

Put $c=1/2$.  Let $a:\mathbb R\to\mathbb R$ be an increasing
$C^2$ Morse-Smale diffeomorphism whose only fixed points are
$-1,0,1$.  We assume that
\[
 a'(-1)=a'(1)=\lambda_1,
 \qquad
 a'(0)=\lambda_1^{-1}
\]
for some $0<\lambda_1<1$.  Thus the points $\pm1$ are attracting,
$0$ is repelling, and
\begin{equation}
\label{eq:product-a-order}
 a(t)>t \quad\text{for }t\in(-\infty,-1)\cup(0,1),
 \qquad
 a(t)<t \quad\text{for }t\in(-1,0)\cup(1,+\infty).
\end{equation}
We assume that there is $M>0$ with
\begin{equation}
\label{eq:product-a-bounds}
 \lambda_1\leq a'(t)\leq \lambda_1^{-1},
 \qquad |a''(t)|\leq M,
 \qquad t\in\mathbb R.
\end{equation}
We additionally require constants $C>0$ and $\lambda\in(0,1)$ such that
\begin{equation}
\label{eq:product-scalar-estimates}
 |(a^n)'(t)|\leq C\lambda^n
 \quad\text{for }|t|>c,
 \qquad
 |(a^{-n})'(t)|\leq C\lambda^n
 \quad\text{for }|t|\leq c,
 \qquad n\geq1.
\end{equation}
These are precisely the generalized $(C,\lambda)$ estimates for the
scalar splitting
\begin{equation}
\label{eq:product-scalar-splitting}
 E_t^s=
 \begin{cases}
  \mathbb R,& |t|>c,\\
  0,& |t|\leq c,
 \end{cases}
 \qquad
 E_t^u=
 \begin{cases}
  0,& |t|>c,\\
  \mathbb R,& |t|\leq c.
 \end{cases}
\end{equation}

Scalar maps satisfying all the preceding assumptions, including the
uniform iterated estimates \eqref{eq:product-scalar-estimates}, can be
chosen as in \cite[Section~3.2]{TikhomirovGH}, giving the same product
example.

Fix $p\in[1,+\infty]$ and put
 $\Bb=\ell^p(\Zz)$.
The norm of a vector $x\in\Bb$ is denoted by $|x|=|x|_p$.  We also use
$|x|_\infty=\sup_j|x_j|$ for the coordinate supremum; for $p<+\infty$,
$|x|_\infty\leq |x|_p$, and for $p=+\infty$ these two norms coincide.
The symbol $\mathcal L(\Bb)$ denotes the Banach space of bounded linear
operators $\Bb\to\Bb$, with operator norm $\|\cdot\|$.
Define
\begin{equation}
\label{eq:product-f}
 f:\Bb\longrightarrow\Bb,
 \qquad
 f(x)_j=a(x_j),
 \qquad j\in\Zz.
\end{equation}

\begin{example}
\label{ex:product}
The coordinatewise formula \eqref{eq:product-f} defines a diffeomorphism
$f\in\DUC(\Bb)$ satisfying generalized hyperbolicity.
\end{example}
\begin{proof}
Apply \cite[Proposition~4\textup{(P4)}]{TikhomirovGH} to the identical
scalar factors $a$, using $a(0)=0$, \eqref{eq:product-a-bounds}, and the
scalar splitting \eqref{eq:product-scalar-splitting} with the estimates
\eqref{eq:product-scalar-estimates}; see also
\cite[Example~20]{TikhomirovGH}.  This gives $f\in\DUC(\Bb)$ and
generalized hyperbolicity for every $p\in[1,+\infty]$, with the same
constants $C,\lambda$ as for $a$.
\end{proof}

\begin{thm}[Structural stability of the infinite product]
\label{thm:product-structural-stability}
There exist constants $d_0>0$ and $M_0>0$, depending only on $a$, such
that the following holds.
Let $\Bb=\ell^p(\Zz)$, let $f$ be defined by
\eqref{eq:product-f}, and let $\Delta:\Bb\to\Bb$ be a bounded
Lipschitz map satisfying
\begin{equation}
\label{eq:product-fg-close}
 |\Delta|_\infty\leq d,
 \qquad \Lip(\Delta)\leq d,
 \qquad 0<d<d_0.
\end{equation}
Put $g:=f+\Delta$.  Then $g$ is a bi-Lipschitz homeomorphism.
The perturbation $g-f$ may couple all coordinates.  There exist
$h,k\in C_b^0(\Bb,\Bb)$ with
$
 |h|_\infty,|k|_\infty\leq M_0d
$
such that, setting $H=\Id_{\Bb}+h$ and $K=\Id_{\Bb}+k$,
$$
 g\circ H=H\circ f, \quad K \circ g = f \circ K, \quad 
 K\circ H=H\circ K=\Id_{\Bb}.
$$
In particular, $H$ and $K$ are mutually inverse homeomorphisms.
\end{thm}

\begin{corr}
\label{cor:product-structural-stability}
For every $p\in[1,+\infty]$, the product map $f$ in
\eqref{eq:product-f} is structurally stable.
\end{corr}

\begin{proof}
Given $\varepsilon>0$, take $M_0,d_0$ from
Theorem~\ref{thm:product-structural-stability} and choose
$0<d<\min\{d_0,\varepsilon/M_0\}$.  If $g\in\DUC(\Bb)$ satisfies
$\norm{f-g}_{C^1}<d$, then $\Delta:=g-f$ has
$\norm\Delta_\infty<d$ and $\Lip(\Delta)<d$.
Theorem~\ref{thm:product-structural-stability} gives mutually inverse
homeomorphisms $H=\Id+h$ and $K=\Id+k$ satisfying the conjugacy
identities and $\norm h_\infty,\norm k_\infty\leq M_0d<\varepsilon$.
\end{proof}

\begin{remark}
\label{rem:comparison-stability-results}
Theorem~\ref{thm:main} applies to arbitrary generalized-hyperbolic
$f\in\DUC(\Bb)$. 
It gives
continuous semiconjugacies in both directions; uniform continuity
and mutual invertibility are not asserted.
Theorem~\ref{thm:product-structural-stability} concerns the specified
Morse-Smale product.
Here the semiconjugacies are mutually inverse
and both maps can be chosen uniformly continuous
(Remark~\ref{rem:product-uniform-conjugacy}).  The passage to conjugacy
uses the compatible normalization available for the product.
\end{remark}

\begin{corr}
\label{rem:product-uncountably-many}
For every $p\in[1,+\infty]$, all fixed points of the map $f$ in
\eqref{eq:product-f} are hyperbolic.  If $1\leq p<+\infty$, its fixed-point set is
countable.
If $p=+\infty$, it has uncountably many fixed points.
For every $p\in[1,+\infty]$, no generalized-hyperbolic splitting of $f$
depends continuously on the base point, and $f$ has no compact global
attractor in the $\ell^p$ norm.
\end{corr}

\begin{proof}
Since the only fixed
points of $a$ are $-1,0,1$, the fixed-point set of $f$ is
\[
 \{q\in\ell^p(\Zz):q_j\in\{-1,0,1\}\text{ for every }j\in\Zz\}.
\]
For $p<+\infty$, such sequences have finite support, so this set is
countably infinite.  For $p=+\infty$, every sequence in
$\{-1,0,1\}^{\Zz}$ is admissible, giving an uncountable set.
At a fixed point $q$, the derivative is the diagonal operator with
multiplier $\lambda_1$ on the coordinates where $q_j\neq0$, and
multiplier $\lambda_1^{-1}$ where $q_j=0$.  Thus every fixed point is
hyperbolic, and its stable subspace consists of the vectors supported on
$\{j:q_j\neq0\}$.
For any generalized-hyperbolic splitting, \textup{(GH3)} gives
$E_q^s$ and $E_q^u$ inside the stable and unstable subspaces of $Df(q)$,
respectively, and \textup{(GH1)} then forces equality.  Hence
$\dim E_0^s=0$ and $\dim E_{e_0}^s=1$.  The rank of a norm-continuous
family of projections is locally constant and $\Bb$ is connected, so no
generalized-hyperbolic splitting of $f$ depends continuously on the base
point.
For a more detailed discussion of the fixed points and their stable
subspaces, see \cite[Section~3.2]{TikhomirovGH}.

Every unit coordinate vector $e_j$ is fixed, so any compact global
attractor would contain all of them.  Since $|e_i-e_j|_p\geq1$ for
$i\neq j$, this contradicts compactness.
\end{proof}

Thus the global structural-stability result above lies outside the
compact-attractor framework considered in \cite{ArrietaCarvalhoTakaessu}.
For oriented shadowing by elementary pseudotrajectories near attractors
in Banach spaces, see \cite{LiPilyugin}.
The proof of Theorem~\ref{thm:product-structural-stability} is given
in Section~\ref{sec:infinite-morse-smale}.

\subsection{Continuous solvability and shadowing for cocycles}
\label{sec:induced-shadowing-results}

Let $\alpha:\Bb\to\Bb$ be a homeomorphism and let
$A:\Bb\to\operatorname{GL}(\Bb)$ be the generator of a linear cocycle
over $\alpha$.  Assume that there exists
$R_A>0$ such that
\[
 \norm{A(x)},\norm{A(x)^{-1}}\leq R_A,
 \qquad x\in\Bb,
\]
and that the maps $x\mapsto A(x)$ and $x\mapsto A(x)^{-1}$ are
uniformly continuous in the operator norm.  Put $A_\alpha^0(x):=\Id$
and, for $n>0$, define
\begin{equation*}
 A_\alpha^n(x)
 :=A(\alpha^{n-1}(x))\cdots A(\alpha(x))A(x),
 \qquad
 A_\alpha^{-n}(x)
 :=A(\alpha^{-n}(x))^{-1}\cdots A(\alpha^{-1}(x))^{-1}.
\end{equation*}

\begin{defin}[Cocycle generalized hyperbolicity]
\label{def:CL-cocycle}
We say that the pair $(\alpha,A)$ satisfies \emph{cocycle generalized
hyperbolicity} if there exist constants $C>0$, $\lambda\in(0,1)$ and a
family of projections
$P_x,Q_x:\Bb\to\Bb$, $x\in\Bb$,
such that, writing
$
 E_x^s:=P_x(\Bb)$,
$
 E_x^u:=Q_x(\Bb)$,
for every $x\in\Bb$ the following conditions hold.
\begin{enumerate}[label=\textup{(C-GH\arabic*)},leftmargin=15mm]
\item \textbf{Splitting.}
$
 P_x+Q_x=\Id$,
$ \norm{P_x},\norm{Q_x}\leq C.
$
Hence, $E_x^s\oplus E_x^u=\Bb$.

\item \textbf{Inclusion.} $
 A(x)E_x^s\subseteq E_{\alpha(x)}^s$,
$ A(x)^{-1}E_{\alpha(x)}^u\subseteq E_x^u.
$ hold.

\item \textbf{Exponential estimates.} For every $n>0$,
\[
 \abs{A_\alpha^n(x)v_x^s}
 \leq C\lambda^n\abs{v_x^s},
 \qquad v_x^s\in E_x^s,
\]
\[
 \abs{A_\alpha^{-n}(x)v_x^u}
 \leq C\lambda^n\abs{v_x^u},
 \qquad v_x^u\in E_x^u.
\]
\end{enumerate}
As in Definition \ref{def:GH}, no finite-dimensionality assumption is imposed on $E_x^u$.
No continuity, uniform continuity, or measurability of the maps
$x\mapsto P_x,Q_x$ is assumed.
\end{defin}

For $f\in\DUC(\Bb)$, the cocycle $A(x)=Df(x)$ satisfies the
assumptions above.  Indeed, \eqref{eq:C1-2} gives the uniform bounds,
\eqref{eq:C1-r1} gives the uniform continuity of $Df$, and
\[
 \norm{Df(x)^{-1}-Df(y)^{-1}}
 \leq R^2\norm{Df(x)-Df(y)}
 \leq R^2r(\abs{x-y}).
\]
Moreover, if $f$ satisfies generalized hyperbolicity, then $(f,Df)$
satisfies cocycle generalized hyperbolicity with the same projections and
the same constants.

Put $X=C_b^0(\Bb,\Bb)$, with the supremum norm.  We first state
bounded continuous solvability of the cohomological equation.

\begin{thm}
\label{thm:cocycle-D-surj}
Assume that $(\alpha,A)$ satisfies cocycle generalized hyperbolicity, with constants
$C,\lambda$.  There exists $K_c=K_c(C,\lambda)>0$ such that, for
every $w\in X$, there exists $v\in X$ satisfying
\begin{equation}
\label{eq:cocycle-D-solve}
 \norm v_\infty\leq K_c\norm w_\infty, \qquad v(\alpha(x))-A(x)v(x)=w(x),\qquad x\in\Bb.
\end{equation}
\end{thm}

On $X$, define the induced linear operator by
\begin{equation}\label{eq:TalphaA}
 (T_{\alpha,A}v)(\alpha(x)):=A(x)v(x).
\end{equation}
It is a bounded linear isomorphism.  The following result concerns its
dynamics on the space of continuous vector fields.

\begin{thm}[Shadowing for the induced operator]
\label{thm:induced-shadowing}
Assume that $(\alpha,A)$ satisfies cocycle generalized hyperbolicity.
Then $T_{\alpha,A}$ has the Lipschitz shadowing property on $X$:
there exists $L_T>0$ such that, for every $d>0$ and
every sequence $(v_n)_{n\in\Zz}$ in $X$ satisfying
$
 \sup_{n\in\Zz}\norm{v_{n+1}-T_{\alpha,A}v_n}_\infty\leq d,
$
there is $v\in X$ such that
$
 \sup_{n\in\Zz}\norm{v_n-T_{\alpha,A}^n v}_\infty\leq L_Td.
$
\end{thm}

We do not establish generalized hyperbolicity of $T_{\alpha,A}$.
The proof is given in Subsection~\ref{sec:finite-cont}.
For spectral characterizations of shadowing and its relation to uniform
expansivity, see \cite{PitukSpectral,DragicevicPituk}.
The distinction between shadowing and generalized hyperbolicity in linear
dynamics is discussed in Subsection~\ref{sec:future-research}.
The Hilbert-space results mentioned there do not directly apply here:
even when $\Bb$ is a Hilbert space,
$X=C_b^0(\Bb,\Bb)$ is not a Hilbert space.

\subsection{Scheme of the proofs}
\label{sec:proof-scheme}

We outline the cohomological approach to Theorems~\ref{thm:main},
\ref{thm:topological-stability},
and~\ref{thm:product-structural-stability}, emphasizing the additional
compatibility needed for conjugacy.
Put $X=C_b^0(\Bb,\Bb)$ and $g=f+\Delta$.  Writing the semiconjugacies
as $H=\Id+h$ and $K=\Id+k$ gives
\[
\begin{aligned}
 h(f(x))-Df(x)h(x)
 &=s_f(x,h(x))+\Delta(x+h(x)),\\
 k(g(x))-Df(x)k(x)
 &=s_f(x,k(x))-\Delta(x).
\end{aligned}
\]
Both equations require bounded continuous solvability of a linear
cohomological equation, followed by a nonlinear correction.  The
reverse equation changes the base from $f$ to $g$, while retaining
the coefficient $Df$.  We therefore develop the linear theory for
cocycles $(\alpha,A)$ with an independent base map and coefficient.
We describe the argument in the order of its presentation.

\medskip
\noindent\textbf{1. Orbitwise estimates and continuous finite-time solutions.}
Generalized hyperbolicity gives uniformly bounded solutions along
each orbit, but these need not depend continuously on the base point.
Michael's selection theorem gives continuous solutions to finite-time
recurrence problems, with bounds independent of the horizon.  As a
by-product, these finite-time estimates yield Lipschitz shadowing for
$T_{\alpha,A}$, using the Bernardes--Peris finite-to-infinite shadowing
result (Theorem~\ref{thm:induced-shadowing}).

\medskip
\noindent\textbf{2. Continuous solvability of cohomological equations.}
Birkhoff averaging of the finite-time solutions gives bounded
continuous approximate solutions of the cohomological equation.
Successive corrections then give exact solutions
(Theorem~\ref{thm:cocycle-D-surj}).  No continuous choice of solution
as a function of the forcing is needed.

\medskip
\noindent\textbf{3. Changing the base for the reverse equation.}
Every orbit of $g$ is a pseudotrajectory of $f$.  Shadowing and uniform
continuity of $Df$ allow us to compare the derivative sequences along
these orbits with those along exact $f$-orbits.  Periodic shadowing
handles periodic orbits.  Robustness for sequences of operators then
gives generalized hyperbolicity of $(g,Df)$, with uniform constants
(Theorem~\ref{thm:g-Df-GH}).  This step relies on the Lipschitz and
periodic shadowing theorems and the robustness lemma of
\cite{TikhomirovGH}, recalled in Subsection~\ref{sec:shadowing-background}.  The cocycle solvability theorem
consequently applies to both cocycles $(f,Df)$ and $(g,Df)$.

\medskip
\noindent\textbf{4. Constructing the semiconjugacies.}
The nonlinear terms in the two semiconjugacy equations have small
Lipschitz constants on a sufficiently small ball.  The bounded linear
solvability estimates therefore allow us to apply the
successive-correction argument of
Lemma~\ref{lem:successive-corrections}, a quantitative
Lyusternik--Graves iteration.  The corrections converge geometrically
and give continuous $h,k$ with norms bounded by a constant times the
perturbation size, proving Theorem~\ref{thm:main}.  In the reverse
equation the perturbation cancels from the nonlinear term, so this
equation requires no regularity of $g$ beyond continuity; this gives
Theorem~\ref{thm:topological-stability}.  Without expansivity,
these two semiconjugacies need not be mutually inverse.  The product
construction addresses this additional difficulty by choosing
compatible normalizations.

\medskip
\noindent\textbf{5. Scalar estimates and normalized Green operators.}
We continue with the proof of structural stability of the infinite-product example.
Uniform scalar estimates and a fixed cutoff give explicit Green
formulas, together with bounds for their dependence on the corrected
trajectory.  These formulas define a bounded linear right inverse
$O_f$ for the forward equation.  For the reverse equation, the
normalization
$
 u=\mathcal S_g(u)w
$
defines a local nonlinear solution operator $u=O_gw$
(Lemmas~\ref{lem:product-Of} and~\ref{lem:product-Og}).  The derivative
and cutoff in this formula are evaluated at $y+u(y)$.  This choice
gives the covariance under composition needed in the final step.
The normalization also ensures uniqueness of the selected solution
and Lipschitz dependence on $w$.

\medskip
\noindent\textbf{6. Pullback compatibility, remainder estimates, and conjugacy.}
We first prove pullback compatibility of the normalized operators
(Lemma~\ref{lem:product-pullback}).  We then establish Lipschitz and
cancellation estimates for the remainder operators $G$ and $R$
(Lemma~\ref{lem:product-remainders}).  The Lipschitz estimates make
$
 h=O_fG(h)$, $k=O_gR(k)$
contraction problems on a small ball, giving the two semiconjugacies.
Combining pullback compatibility with the cancellation estimates
yields
\[
\begin{aligned}
 |h+k\circ H|_\infty
 &\leq q_1|h+k\circ H|_\infty,\qquad q_1<1,\\
 |k+h\circ K|_\infty
 &\leq q_2|k+h\circ K|_\infty,\qquad q_2<1.
\end{aligned}
\]
Both composition defects vanish, so $K\circ H=H\circ K=\Id$.
These estimates prove invertibility without assuming it in the
pullback argument.  They are uniform in $p$, including $p=\infty$,
and establish Theorem~\ref{thm:product-structural-stability}.

\subsection{Future research}
\label{sec:future-research}

\textbf{From generalized hyperbolicity to structural stability.}
Does every generalized-hyperbolic $f\in\DUC(\Bb)$ satisfy structural
stability in the sense of Definition~\ref{def:stability}?
Theorem~\ref{thm:main} supplies two continuous semiconjugacies, but
their mutual invertibility remains open.  Extending the product argument
requires compatible solution operators, as in the covariance and pullback
results of Subsection~\ref{sec:product-pullback}, or another mechanism
yielding the inverse identities.  A natural test case is a coordinatewise
product of planar Morse-Smale diffeomorphisms with saddles.  There the
stable direction varies from point to point, the derivative is diagonal
in no fixed frame, and trajectories may stay near a saddle for
arbitrarily long times, so neither the scalar estimates nor the scalar
cutoff of Section~\ref{sec:infinite-morse-smale} have direct analogues.

\textbf{A converse characterization.}
It is also natural to ask whether structural stability in this global
category implies generalized hyperbolicity, as is the case for
diffeomorphisms of smooth closed manifolds, and whether the answer
depends on the Banach space.  Recent announcements in linear dynamics
show that the analogous converse from shadowing is sensitive to the
category.  Pituk \cite{Pituk} proves that shadowing
is equivalent to generalized hyperbolicity for invertible bounded
operators on separable complex Hilbert spaces, but constructs a
counterexample on a complex Banach space; see also \cite{MessaoudiEtal}.
The distinction involves surjectivity versus the existence of bounded
linear right inverses.  These results concern shadowing and motivate
examining the geometry of the phase space in a structural-stability
characterization.  They do not settle the nonlinear
converse posed here.

\textbf{Beyond the global invertible category.}
The class $\DUC(\Bb)$ contains nonlinear shifts and infinite products
of Morse--Smale systems \cite{TikhomirovGH}, but does not directly
cover the usual semiflows of delay and parabolic evolution equations
discussed in Subsection~\ref{sec:historical-overview}.
Their time maps are generally noninvertible on the natural phase spaces,
and their regularity assumptions differ from the global derivative
bounds and uniform continuity required here; see also
\cite[Remark~5]{TikhomirovGH}.  An extension must therefore address
both the use of backward trajectories and the relevant phase spaces,
regularity, and notion of stability.
Neutral delay differential equations may offer a natural setting in which
the full scope of generalized hyperbolicity becomes relevant: 
infinite-dimensional unstable subspaces can occur
\cite{HaleVerduynLunel,HenryNeutral}.

Even for noninvertible maps of finite-dimensional compact manifolds,
a general characterization of structural stability on the original
phase space by Axiom~A and transversality remains incomplete.
The inverse-limit stability theorem of Berger and Kocsard
\cite{BergerKocsard} emphasizes the distinction between stability of
full-orbit dynamics and stability on the original space.

\section{Dynamical and functional-analytic preliminaries}
\label{sec:Michael}

We now turn to the proofs.  We first recall the shadowing and
robustness results used in the cocycle arguments, then collect the
functional-analytic tools for constructing continuous solutions.

\subsection{Shadowing and robustness}
\label{sec:shadowing-background}

We recall the shadowing terminology; for general background, see
\cite{PalmerShadowing,PilyuginShadowing}.
Let $F:X\to X$ be a homeomorphism of a Banach space.  For $d>0$, a
sequence $(y_k)_{k\in I}$, where $I$ is a finite integer interval or
$\Zz$, is a \emph{$d$-pseudotrajectory} if
\[
 \norm{y_{k+1}-F(y_k)}\leq d,
 \qquad k,k+1\in I.
\]
It is \emph{$\varepsilon$-shadowed} if there exists $z\in X$ such that
$\sup_{k\in I}\norm{y_k-F^k(z)}\leq\varepsilon$.
The map $F$ has \emph{finite shadowing} if, for every $\varepsilon>0$,
there exists $d>0$ such that every finite $d$-pseudotrajectory is
$\varepsilon$-shadowed.  Replacing finite intervals by $\Zz$ defines
the \emph{shadowing property}.  In the Lipschitz versions, there are
constants $M,d_0>0$ such that every $d$-pseudotrajectory of the
corresponding type is $Md$-shadowed for $0<d<d_0$.

\begin{lem}
\label{lem:finite-infinite-shadowing}
Let $T:X\to X$ be a bounded linear isomorphism of a Banach space.
If $T$ has finite shadowing, then it has shadowing.
More precisely, if $T$ has finite Lipschitz shadowing with constant
$L>0$, then, for every $d>0$ and every $d$-pseudotrajectory
$(y_k)_{k\in\Zz}$, there exists $z\in X$ such that
\[
 \sup_{k\in\Zz}\norm{y_k-T^kz}\leq2Ld.
\]
\end{lem}

The qualitative statement is due to Bernardes and Peris
\cite[Theorem~1]{BernardesPeris}; the quantitative form with constant
$2L$ is recorded, with its proof, in \cite[Appendix~A]{TikhomirovGH}.

\begin{thm}
\label{thm:GH-shadowing}
Let $f\in\DUC(\Bb)$ satisfy generalized hyperbolicity in the sense of
Definition~\ref{def:GH}.  There exist constants $M,d_0>0$ such that,
for every $0<d<d_0$, the following hold.
\begin{enumerate}[label=\textup{(\roman*)},leftmargin=10mm]
\item \textbf{Lipschitz shadowing.}
For every $d$-pseudotrajectory $(y_k)_{k\in\Zz}$ of $f$, there exists
$x\in\Bb$ such that
\[
 \sup_{k\in\Zz}\abs{y_k-f^k(x)}\leq Md.
\]
\item \textbf{Lipschitz periodic shadowing.}
For every integer $m\geq1$ and every $d$-pseudotrajectory
$(y_k)_{k\in\Zz}$ of $f$ with $y_{k+m}=y_k$ for all $k\in\Zz$,
there exists $x\in\Bb$ such that
\[
 f^m(x)=x,
 \qquad
 \sup_{k\in\Zz}\abs{y_k-f^k(x)}\leq Md.
\]
\end{enumerate}
\end{thm}

This is \cite[Theorem~11\,(M1)--(M2)]{TikhomirovGH}.  The constants
are independent of the pseudotrajectory and its period.  No
regularity of the generalized-hyperbolic splitting is required.

We also need robustness of generalized hyperbolicity for sequences of
operators.  For a sequence $\mathcal A=(A_k)_{k\in\Zz}$ of bounded
linear isomorphisms of $\Bb$ with uniformly bounded operator and inverse
norms, generalized hyperbolicity with constants $C>0$, $\lambda\in(0,1)$
means that there exist bounded projections
$P_k,Q_k$, with $E_k^s=P_k(\Bb)$ and $E_k^u=Q_k(\Bb)$, satisfying
\[
 P_k+Q_k=\Id,\qquad \norm{P_k},\norm{Q_k}\leq C,
\qquad
 A_kE_k^s\subseteq E_{k+1}^s,\qquad
 A_k^{-1}E_{k+1}^u\subseteq E_k^u,
\]
and, for every $k\in\Zz$ and $n\geq0$,
\[
\begin{aligned}
 \abs{A_{k+n-1}\cdots A_kv^s}
 &\leq C\lambda^n\abs{v^s}, &&v^s\in E_k^s,\\
 \abs{A_{k-n}^{-1}\cdots A_{k-1}^{-1}v^u}
 &\leq C\lambda^n\abs{v^u}, &&v^u\in E_k^u.
\end{aligned}
\]
An $m$-periodic sequence satisfies
\emph{periodic generalized hyperbolicity} if these projections can be
chosen $m$-periodic.

\begin{lem}
\label{lem:sequence-GH-robustness}
Assume that $\mathcal A=(A_k)_{k\in\Zz}$ satisfies generalized
hyperbolicity with constants $C>0$, $\lambda\in(0,1)$ and
$
 \norm{A_k},\norm{A_k^{-1}}\leq R$.
For every $\lambda'\in(\lambda,1)$ there exist $C_1>0$ and
$\varepsilon>0$, depending only on $C,\lambda,\lambda',R$, such that
the following hold.
\begin{enumerate}[label=\textup{(\roman*)},ref=(\roman*),leftmargin=10mm]
\item \label{item:sequence-robustness-general}
Every sequence $\mathcal B=(B_k)_{k\in\Zz}$ of bounded linear
operators satisfying
\begin{equation}
\label{eq:sequence-robustness-close}
 \norm{B_k-A_k}<\varepsilon,\qquad k\in\Zz,
\end{equation}
satisfies generalized hyperbolicity with constants $C_1,\lambda'$.
\item \label{item:sequence-robustness-periodic}
If, in addition, $\mathcal A$ is $m$-periodic and satisfies periodic
generalized hyperbolicity, then every $m$-periodic sequence $\mathcal B$
satisfying \eqref{eq:sequence-robustness-close} satisfies periodic
generalized hyperbolicity with the same constants $C_1,\lambda'$.
\end{enumerate}
The constants are independent of the sequences and of the period $m$.
\end{lem}

This is \cite[Lemma~35\,(Robust1)--(Robust2)]{TikhomirovGH}.

\subsection{Functional-analytic background}
\label{sec:functional-analytic-background}

We begin with the selection theorem in the form used below.

\begin{defin}
\label{def:lsc}
Let $X,Y$ be topological spaces.  A set-valued map
$\Theta:X\to2^Y$ is \emph{lower semicontinuous} if, for every open set
$U\subset Y$, the set
$
 \{x\in X:\Theta(x)\cap U\neq\varnothing\}
$
is open in $X$.
\end{defin}

\begin{thm}[Michael's continuous selection]
\label{thm:Michael}
Let $X$ be a paracompact topological space and let $Y$ be a Banach space.
Suppose that $\Theta:X\to2^Y$ is lower semicontinuous and for every
$x\in X$, the set $\Theta(x)$ is a nonempty closed convex subset of $Y$.
Then $\Theta$ admits a continuous selection: there exists a continuous map
$a:X\to Y$ such that
$a(x)\in\Theta(x)$ for $x\in X$.
\end{thm}

This is \cite[Theorem~3.2'']{MichaelI}; see also \cite{MichaelII}.  We
shall apply it with $X=Y=\Bb$.  This is legitimate because every metric
space is paracompact \cite{Engelking}.

The mechanism behind the next finite-in-time lemma is
similar to arguments used in parametric optimization.  Strict feasibility is a
Slater-type condition: convexity allows one to approximate an arbitrary
feasible point by strictly feasible points, and continuity then preserves
the finitely many strict inequalities under a small change of the parameter.
See \cite{Hogan} for point-to-set maps determined by inequalities and
\cite{DinhGobernaLopez} for the broader feasible-set stability theory in
locally convex spaces.  We have not found a reference for exactly the
formulation below, where the parameter space is merely topological and the
decision and range spaces are Banach spaces, so we include the short proof.

\begin{lem}
\label{lem:strict}
Let $X$ be a topological space, let $Y,Z$ be Banach spaces, let
$N\in\mathbb N$, and let
$V_k:X\times Y\longrightarrow Z$, $k=0,\ldots,N$,
be continuous maps such that $y\mapsto V_k(x,y)$ is affine for every
$x\in X$.  Let $M>0$ and $\alpha>0$.  Assume that, for every $x\in X$,
there exists $y^*=y^*(x)\in Y$ satisfying
\begin{equation}
\label{eq:strict-feasible}
 \norm{V_k(x,y^*)}\leq M-\alpha,
 \qquad k=0,\ldots,N.
\end{equation}
We refer to \eqref{eq:strict-feasible} as a \emph{uniform Slater-type
condition}: one point satisfies every constraint with the same positive
margin $\alpha$.
Define
\begin{equation*}
 \Theta(x):=
 \{y\in Y:\norm{V_k(x,y)}\leq M
       \text{ for }k=0,\ldots,N\}.
\end{equation*}
Then every $\Theta(x)$ is nonempty, closed and convex, and the map
$\Theta:X\to2^Y$ is lower semicontinuous.
\end{lem}

\begin{proof}
Nonemptiness follows from \eqref{eq:strict-feasible}.  For a fixed $x$,
each set
$\{y\in Y:\norm{V_k(x,y)}\leq M\}$
is closed because $V_k(x,\cdot)$ is continuous, and it is convex because
$V_k(x,\cdot)$ is affine and the closed ball of radius $M$ in $Z$ is
convex.  Their finite intersection $\Theta(x)$ is therefore closed and
convex.

It remains to prove lower semicontinuity.  Fix $x_0\in X$ and an open set
$U\subset Y$ such that $\Theta(x_0)\cap U\neq\varnothing$.  Choose
$y_0\in\Theta(x_0)\cap U$.  By the uniform Slater-type condition,
there is a strictly feasible point $y^*$ at $x_0$ satisfying
\eqref{eq:strict-feasible}.  For
$t\in(0,1)$ put
$y_t:=ty_0+(1-t)y^*$.
For $t$ sufficiently close to $1$, one has $y_t\in U$.  By affinity,
\[
 V_k(x_0,y_t)
 =tV_k(x_0,y_0)+(1-t)V_k(x_0,y^*),
\]
and hence
\begin{align*}
 \norm{V_k(x_0,y_t)}
 \leq tM+(1-t)(M-\alpha)
 =M-(1-t)\alpha<M,
 \qquad k=0,\ldots,N.
\end{align*}
Thus all the defining inequalities are strict at $(x_0,y_t)$.  Since
there are only finitely many of them and all $V_k$ are continuous, there
exists a neighborhood $W$ of $x_0$ such that
\[
 \norm{V_k(x,y_t)}<M,
 \qquad x\in W,
 \quad k=0,\ldots,N.
\]
Consequently $y_t\in\Theta(x)\cap U$ for every $x\in W$.  By
Definition~\ref{def:lsc}, $\Theta$ is lower semicontinuous.
\end{proof}

\begin{remark}
The finiteness of the horizon is essential in this argument: continuity
provides a common neighborhood for finitely many strict inequalities.  The
same proof does not directly give lower semicontinuity for a fiber defined
by infinitely many inequalities.
\end{remark}
The following nonlinear lemma is a special quantitative form of the
Lyusternik--Graves argument \cite{Lyusternik, Graves}.  
We include the short construction since it uses only quantitative bounded
solvability of the linear equation.  
The lemma will be applied to both semiconjugacy equations below.
\begin{lem}
\label{lem:successive-corrections}
Let $\mathcal X,\mathcal Y$ be Banach spaces and let
$\mathcal L:\mathcal X\to\mathcal Y$ be a bounded linear operator.  Assume
that there exists $K>0$ such that, for every $w\in\mathcal Y$, one can choose
$q\in\mathcal X$ satisfying
\begin{equation}
\label{eq:abstract-preimage}
 \mathcal Lq=w,
 \qquad
 \norm q\leq K\norm w.
\end{equation}
Let $\rho>0$, and let
$\mathcal N:\{u\in\mathcal X:\norm u\leq\rho\}\longrightarrow\mathcal Y$
be Lipschitz with constant $\mu$.  If
\begin{equation}
\label{eq:successive-smallness}
 \theta:=K\mu<1,
 \qquad
 \frac{K\norm{\mathcal N(0)}}{1-\theta}\leq\rho,
\end{equation}
then there exists $u\in\mathcal X$ such that
 $\mathcal Lu=\mathcal N(u)$,
 $\norm u\leq \frac{K\norm{\mathcal N(0)}}{1-\theta}\leq\rho$.
No linear or continuous choice of the preimages in
\eqref{eq:abstract-preimage} is required.
\end{lem}

\begin{proof}
Set $u_0:=0$.  Choose $q_0\in\mathcal X$ according to
\eqref{eq:abstract-preimage} with $\mathcal Lq_0=\mathcal N(u_0)$,
and put $u_1:=q_0$.  For $j\geq1$, once $u_{j-1}$ and $u_j$ have
been defined and lie in the ball of radius $\rho$, choose
$q_j\in\mathcal X$ according to \eqref{eq:abstract-preimage} with
\begin{equation}
\label{eq:abstract-correction}
 \mathcal Lq_j=\mathcal N(u_j)-\mathcal N(u_{j-1}),
\end{equation}
and put $u_{j+1}:=u_j+q_j$.
We now show that the construction continues for all $j\geq0$ and that
\begin{equation}
\label{eq:abstract-geometric}
 \norm{q_j}\leq K\norm{\mathcal N(0)}\theta^j,
 \qquad
 \norm{u_{j+1}}
 \leq K\norm{\mathcal N(0)}\sum_{i=0}^{j}\theta^i
 \leq\rho.
\end{equation}
The case $j=0$ follows from the choice of $q_0$ and
\eqref{eq:successive-smallness}.  Suppose the estimates hold through
$j-1$.  Then $u_{j-1}$ and $u_j$ lie in the ball of radius $\rho$,
so $q_j$ and $u_{j+1}$ can be defined.  By the choice of $q_j$,
the Lipschitz estimate, and $u_j-u_{j-1}=q_{j-1}$,
\[
 \norm{q_j}
 \leq K\norm{\mathcal N(u_j)-\mathcal N(u_{j-1})}
 \leq K\mu\norm{q_{j-1}}
 \leq K\norm{\mathcal N(0)}\theta^j.
\]
Since $u_{j+1}=\sum_{i=0}^{j}q_i$, summing these bounds and using
\eqref{eq:successive-smallness} proves the remaining inequalities in
\eqref{eq:abstract-geometric}.  Thus every step is well-defined.
Moreover, induction using \eqref{eq:abstract-correction} gives
\begin{equation}
\label{eq:abstract-approximation}
 \mathcal Lu_{j+1}=\mathcal N(u_j),
 \qquad j\geq0.
\end{equation}

The series $\sum_{j\geq0}q_j$
converges to some $u\in\mathcal X$, and
$
 \norm u\leq\frac{K\norm{\mathcal N(0)}}{1-\theta}\leq\rho.
$
Since $\mathcal L$ is bounded and $\mathcal N$ is Lipschitz, passing to
the limit in \eqref{eq:abstract-approximation} gives
$\mathcal Lu=\mathcal N(u)$.
\end{proof}

\section{Cocycles}
\label{sec:linear-input}

\subsection{Orbitwise estimates and shadowing for the induced operator}
\label{sec:finite-cont}

Recall the standing assumptions on the cocycle $(\alpha,A)$ and
Definition~\ref{def:CL-cocycle} from
Subsection~\ref{sec:induced-shadowing-results}.  We begin with the
orbitwise estimate underlying all constructions.

\begin{lem}
\label{lem:cocycle-SBS}
Assume that $(\alpha,A)$ satisfies cocycle generalized hyperbolicity with
constants $C>0$ and $\lambda\in(0,1)$.  There exists
$L_c=L_c(C,\lambda)>0$ such that, for every $x\in\Bb$, every finite or
infinite integer interval $I$, and every bounded sequence
$(w_{k+1})_{k\in\widetilde I}$, where
$\widetilde I:=\{k\in\Zz:k,k+1\in I\}$, there exists a sequence
$(v_k)_{k\in I}$ satisfying $ \sup_{k\in I}\abs{v_k}
 \leq L_c\sup_{k\in\widetilde I}\abs{w_{k+1}}$ and
\[
 v_{k+1}=A(\alpha^k(x))v_k+w_{k+1},
 \qquad k\in\widetilde I.
\]
\end{lem}

\begin{remark}
The bounded-solvability approach goes back to Perron \cite{Perron}
and the classical work of Maizel and Pliss
\cite{Maizel,PlissBounded}; see Coppel \cite{Coppel} for a broader
treatment. For nonautonomous evolution equations,
the relation between exponential dichotomy and bounded mild
solutions is studied in \cite{LatushkinRandolphSchnaubelt}.
The proof below uses the Perron construction in
\cite[proof of Lemma~31]{TikhomirovGH}.
\end{remark}      

\begin{proof}
Fix $x\in\Bb$ and first take $I=\Zz$.  Put $x_j:=\alpha^j(x)$ and
define $v_k$ by the Perron sums
\[
 v_k:=
 \sum_{j\leq k}A_\alpha^{k-j}(x_j)P_{x_j}w_j
 -
 \sum_{j>k}A_\alpha^{k-j}(x_j)Q_{x_j}w_j,
 \qquad k\in\Zz.
\]
The convergence of these sums and the identity
$v_{k+1}=A(x_k)v_k+w_{k+1}$ follow exactly as in
\cite[Lemma~31]{TikhomirovGH}, with derivative iterates
replaced by cocycle iterates.  The projection bounds and exponential
estimates give
\[
 \sup_{k\in\Zz}\abs{v_k}
 \leq\left(C+2C^2\sum_{n=1}^{\infty}\lambda^n\right)
       \sup_{j\in\Zz}\abs{w_j}
 =\left(C+\frac{2C^2\lambda}{1-\lambda}\right)
       \sup_{j\in\Zz}\abs{w_j}.
\]
Thus we may take $L_c:=C+2C^2\lambda/(1-\lambda)$.

For a general interval $I$, extend the forcing by setting
$w_{k+1}:=0$ for $k\notin\widetilde I$ and restrict the resulting
solution to $I$.  Neither operation increases the supremum norm,
so the same constant works independently of $x$ and $I$.
\end{proof}

Recall that
$X=C_b^0(\Bb,\Bb)$, with norm $\norm v_\infty:=\sup_{x\in\Bb}\abs{v(x)}$,
and recall the induced operator
$T_{\alpha,A}$
defined by \eqref{eq:TalphaA}.
It is a bounded linear isomorphism, with
$
 (T_{\alpha,A}^{-1}v)(x)=A(x)^{-1}v(\alpha(x)).
$

If $x\mapsto P_x,Q_x$ were continuous, their pointwise action on $X$
would define bounded complementary projections and a generalized-hyperbolic
splitting for $T_{\alpha,A}$.  The invariance conditions would still be
inclusions, so this would not imply ordinary hyperbolicity \cite{TikhomirovGH}. Without continuity assumptions on the splitting, these pointwise projections need not preserve $X$.  We instead prove the
required finite-horizon property directly.

\begin{lem}
\label{lem:finite-cont}
Assume that $(\alpha,A)$ satisfies cocycle generalized hyperbolicity with
constants $C,\lambda$, and let $L_c$ be as in
Lemma~\ref{lem:cocycle-SBS}.  For every integer $N\geq1$ and every
collection
$
 w_1,\ldots,w_N\in X
$
there exist $v_0,\ldots,v_N\in X$ such that $ \max_{0\leq k\leq N}\norm{v_k}_\infty
 \leq 2L_c\max_{1\leq j\leq N}\norm{w_j}_\infty$ and
\begin{equation*}
 v_{k+1}=T_{\alpha,A}v_k+w_{k+1},
 \qquad k=0,\ldots,N-1.
\end{equation*}

The constant $L_c$ is independent of $N$ and of the functions $w_j$.
\end{lem}

\begin{proof}
Fix $N\geq1$ and $w_1,\ldots,w_N\in X$.  Put
$
 m:=\max_{1\leq j\leq N}\norm{w_j}_\infty.
$
The case $m=0$ is immediate, so after rescaling assume that $m=1$.
For $x,y\in\Bb$, define
$
 V_0(x,y):=y
$
and, recursively,
\begin{equation}
\label{eq:cocycle-Vk}
 V_{k+1}(x,y)
 :=A(\alpha^k(x))V_k(x,y)
   +w_{k+1}(\alpha^{k+1}(x)),
 \qquad k=0,\ldots,N-1.
\end{equation}
Every $V_k:\Bb\times\Bb\to\Bb$ is continuous and is affine in the second
variable.
For a fixed $x$, Lemma~\ref{lem:cocycle-SBS}, applied with
$I=\{0,\ldots,N\}$, provides an initial vector $y^*(x)$ such that
\[
 \abs{V_k(x,y^*(x))}\leq L_c,
 \qquad k=0,\ldots,N.
\]
Define
\[
 \Theta_N(x):=
 \{y\in\Bb:\abs{V_k(x,y)}\leq2L_c
        \text{ \; for \; }k=0,\ldots,N\}.
\]
The preceding estimate gives strict feasibility with margin $L_c$.
Lemma~\ref{lem:strict} and Theorem~\ref{thm:Michael} therefore give a
continuous selection
$a_N:\Bb\longrightarrow\Bb$ with $a_N(x)\in\Theta_N(x)$.
For $k=0,\ldots,N$, define $v_k\in X$ by
 $v_k(\alpha^k(x)):=V_k(x,a_N(x))$.
This is well-defined and continuous because $\alpha$ is a homeomorphism,
and
$
 \max_{0\leq k\leq N}\norm{v_k}_\infty\leq2L_c.
$
Equation \eqref{eq:cocycle-Vk} gives
\[
 v_{k+1}=T_{\alpha,A}v_k+w_{k+1},
 \qquad k=0,\ldots,N-1.
\]
Undoing the normalization gives the finite-horizon bound with right-hand
side $2L_cm$.
\end{proof}

\begin{remark}
\label{rem:partition-of-unity}
The continuous selection $a_N$ can also be obtained directly from a
partition of unity, without Theorem~\ref{thm:Michael}.  For $z\in\Bb$, the set
\[
 W_z:=\{x\in\Bb:\abs{V_k(x,y^*(z))}<2L_c\text{ for }k=0,\ldots,N\}
\]
is open, because only finitely many continuous maps are involved, and it
contains $z$, since $\abs{V_k(z,y^*(z))}\leq L_c$.  Since $\Bb$ is a metric
space, hence paracompact, there is a partition of unity \cite{Engelking}
$\{\varphi_\gamma\}_{\gamma\in\Gamma}$ subordinate to the open cover
$\{W_z\}_{z\in\Bb}$: the supports of the functions $\varphi_\gamma$ form a
locally finite family, and
$\operatorname{supp}\varphi_\gamma\subset W_{z_\gamma}$ for some
$z_\gamma\in\Bb$.  The index set $\Gamma$ is arbitrary and may be
uncountable.  Put
\[
 a_N(x):=\sum_{\gamma\in\Gamma}\varphi_\gamma(x)\,y^*(z_\gamma),
 \qquad x\in\Bb.
\]
Near each point only finitely many terms of the sum are nonzero, so $a_N$
is well defined and continuous.  If $\varphi_\gamma(x)\neq0$, then
$x\in W_{z_\gamma}$.  Since every $V_k(x,\cdot)$ is affine,
\[
 \abs{V_k(x,a_N(x))}
 =\abs{\sum_{\gamma\in\Gamma}\varphi_\gamma(x)V_k(x,y^*(z_\gamma))}
 <2L_c,
 \qquad k=0,\ldots,N,
\]
and hence $a_N(x)\in\Theta_N(x)$.  
This argument is similar to the first step of the proof of
Theorem~\ref{thm:Michael}, in which a partition of unity gives approximate
selections.
\end{remark}

\begin{proof}[Proof of Theorem~\ref{thm:induced-shadowing}]
Lemma~\ref{lem:finite-cont}, applied to the errors of a finite
pseudotrajectory, gives a correction whose norm is at most $2L_c$ times the
error bound and whose subtraction leaves an exact trajectory.
Thus $T_{\alpha,A}$ has finite Lipschitz shadowing with constant
$2L_c$.  Since it is a bounded linear isomorphism of the Banach space
$X$, Lemma~\ref{lem:finite-infinite-shadowing} gives the claimed
two-sided Lipschitz shadowing estimate with $L_T=4L_c$.
\end{proof}

\subsection{Cohomological equations for cocycles}
\label{sec:cohomological}

For a $C^1$-diffeomorphism $F:M\to M$ of a smooth closed manifold,
the Mather operator \cite{Mather} acts on the space
$C^0(M,TM)$ of continuous vector fields by
$
 (F_*v)(F(x))=DF(x)v(x).
$
Linearizing the conjugacy equation leads to $I-F_*$; its surjectivity
on continuous vector fields is called \emph{infinitesimal stability}
\cite{Robbin,Robinson,ManeInfinitesimal,ManeAS}.
The role of analytic strong transversality and shadowing in
Ma\~n\'e's theorem is discussed in \cite{PilyuginMane}.
Related regularity questions for Anosov perturbations and Liv\v{s}ic
cohomological equations are studied in \cite{LlaveMarcoMoriyon}.

For a cocycle $(\alpha,A)$ satisfying the standing assumptions of
Subsection~\ref{sec:induced-shadowing-results}, define on $X=C_b^0(\Bb,\Bb)$ the bounded
linear cohomological operator
\begin{equation*}
 \mathcal D_{\alpha,A}:X\longrightarrow X,
 \qquad
 (\mathcal D_{\alpha,A}v)(x):=v(\alpha(x))-A(x)v(x).
\end{equation*}
By the definition of $T_{\alpha,A}$,
\begin{equation}
\label{eq:cocycle-D-T}
 (I-T_{\alpha,A})v
 =(\mathcal D_{\alpha,A}v)\circ\alpha^{-1}.
\end{equation}
We now prove bounded continuous solvability of
$\mathcal D_{\alpha,A}v=w$ with a uniform norm bound, even when the
cocycle splitting is discontinuous (Theorem~\ref{thm:cocycle-D-surj}).

We begin with an abstract Birkhoff-averaging lemma.  It avoids any use of
compactness or passage to a convergent subsequence of finite-horizon
solutions. For a related argument, see \cite{MessaoudiEtal}.

\begin{lem}[Finite-horizon Birkhoff error halving]
\label{lem:Birkhoff}
Let $X$ be a Banach space and let $T:X\to X$ be a bounded linear operator.
Assume that there exists $M_0>0$ such that, for every $w\in X$ and every
integer $N\geq1$, one can find $v_0,\ldots,v_N\in X$ satisfying
\begin{equation}
\label{eq:abstract-rec}
 v_{k+1}=Tv_k+w,
 \qquad k=0,\ldots,N-1,
\end{equation}
and
$
 \max_{0\leq k\leq N}\norm{v_k}
 \leq M_0\norm w.
$
Then, for every $w\in X$, there exists $u\in X$ such that
\begin{equation*}
 u-Tu=w,
 \qquad
 \norm u\leq2M_0\norm w.
\end{equation*}
No linear or continuous dependence of $u$ on $w$ is asserted.
\end{lem}

\begin{proof}
Choose once and for all an integer $N>4M_0$, and put
$\kappa:=\frac{2M_0}{N}<\frac12$.
Set $w_0:=w$, apply the hypothesis to $w_0$, and choose
$v_0^{(0)},\ldots,v_N^{(0)}$.  Define the Birkhoff average
$
 u_0:=\frac1N\sum_{k=0}^{N-1}v_k^{(0)}.
$
Using \eqref{eq:abstract-rec},
\begin{align*}
 u_0-Tu_0
 =\frac1N\left(
     \sum_{k=0}^{N-1}v_k^{(0)}
     -\sum_{k=0}^{N-1}(v_{k+1}^{(0)}-w_0)
   \right)
 =w_0+\frac1N\bigl(v_0^{(0)}-v_N^{(0)}\bigr).
\end{align*}
Set
 $w_1:=\frac1N\bigl(v_N^{(0)}-v_0^{(0)}\bigr)$.
Then
\[
 u_0-Tu_0=w_0-w_1,
 \qquad
 \norm{u_0}\leq M_0\norm{w_0},
 \qquad
 \norm{w_1}\leq\kappa\norm{w_0}.
\]

We now define the remaining terms inductively, keeping the same $N$.
Suppose that $w_j$ has been defined for some $j\geq1$.  Apply the
hypothesis of Lemma \ref{lem:Birkhoff} with forcing $w_j$ to choose
$v_0^{(j)},\ldots,v_N^{(j)}\in X$ satisfying
\[
 v_{k+1}^{(j)}=Tv_k^{(j)}+w_j,
 \qquad k=0,\ldots,N-1,
 \qquad
 \max_{0\leq k\leq N}\norm{v_k^{(j)}}\leq M_0\norm{w_j}.
\]
Define
\[
 u_j:=\frac1N\sum_{k=0}^{N-1}v_k^{(j)},
 \qquad
 w_{j+1}:=\frac1N\bigl(v_N^{(j)}-v_0^{(j)}\bigr).
\]
Thus $u_j$ is the correction obtained by averaging the finite
recurrence solution with forcing $w_j$, and $w_{j+1}$ is its
remaining error.  The same calculation as for $j=0$ gives
\begin{equation}
\label{eq:telescoping}
 u_j-Tu_j=w_j-w_{j+1},
 \qquad j\geq0.
\end{equation}
Moreover, for every $j\geq0$,
$$
 \norm{u_j}
 \leq\frac1N\sum_{k=0}^{N-1}\norm{v_k^{(j)}}
 \leq M_0\norm{w_j},\qquad
 \norm{w_{j+1}}
 \leq\frac{\norm{v_N^{(j)}}+\norm{v_0^{(j)}}}{N}
 \leq\kappa\norm{w_j}.
$$
Starting from $w_0=w$, induction therefore gives
\[
 \norm{w_j}\leq\kappa^j\norm w,
 \qquad
 \norm{u_j}\leq M_0\kappa^j\norm w,
 \qquad j\geq0.
\]
The series
$u:=\sum_{j=0}^{\infty}u_j$
converges in $X$, and
\[
 \norm u
 \leq \frac{M_0}{1-\kappa}\norm w
 \leq2M_0\norm w.
\]
For every integer $m\geq0$, summing \eqref{eq:telescoping} gives
\[
 (I-T)\sum_{j=0}^{m}u_j
 =\sum_{j=0}^{m}(w_j-w_{j+1})
 =w-w_{m+1}.
\]
Since $\sum_{j=0}^{m}u_j\to u$, $w_{m+1}\to0$, and $I-T$ is bounded,
passing to the limit gives $(I-T)u=w$.
\end{proof}

\begin{proof}[Proof of Theorem~\ref{thm:cocycle-D-surj}]
Let $L_c$ be the constant from Lemma~\ref{lem:cocycle-SBS}.
Apply Lemma~\ref{lem:finite-cont} with
$w_1=\cdots=w_N=\varphi$.  Its conclusion verifies the hypothesis of
Lemma~\ref{lem:Birkhoff} for $T=T_{\alpha,A}$ with $M_0=2L_c$.
Hence, for every $\varphi\in X$, there exists $v\in X$
satisfying
\[
 v-T_{\alpha,A}v=\varphi,
 \qquad
 \norm v_\infty\leq4L_c\norm\varphi_\infty.
\]
Finally, given $w\in X$, take $\varphi=w\circ\alpha^{-1}$ in the preceding
construction.  Identity~\eqref{eq:cocycle-D-T} gives
$\mathcal D_{\alpha,A}v=w$.  Since $\alpha$ is a homeomorphism,
$\norm{w\circ\alpha^{-1}}_\infty=\norm w_\infty$.  This gives \eqref{eq:cocycle-D-solve} with $K_c=4L_c$.
\end{proof}

\begin{remark}
The cocycle projections enter the construction only through the
orbitwise bounds of Lemma~\ref{lem:cocycle-SBS}.  The finite recurrence
maps involve only continuous data: the base map, cocycle generator, and
forcing.  Michael's theorem therefore provides continuous selections of
initial vectors.  Averaging and successive corrections preserve continuity in the base
point.  The general semiconjugacy construction uses only the resulting
solvability estimate; no linear or continuous dependence of the solution
on the forcing is required.
\end{remark}

\begin{remark}
\label{rem:cocycle-continuity-assumptions}
For Theorems~\ref{thm:cocycle-D-surj} and~\ref{thm:induced-shadowing},
continuity of $A$ in the operator norm suffices; uniform continuity is
used only in the change of base (Theorem~\ref{thm:g-Df-GH}).
\end{remark}

For the first application, take $(\alpha,A)=(f,Df)$ and write
\begin{equation*}
 \mathcal D_f:X\longrightarrow X,
 \qquad
 (\mathcal D_fv)(x):=v(f(x))-Df(x)v(x).
\end{equation*}

\begin{corr}
\label{prop:D-surj}
Let $f\in\DUC(\Bb)$ satisfy generalized hyperbolicity, and let $L$ be the
constant from Lemma~\ref{lem:cocycle-SBS} for the cocycle $(f,Df)$.  Then,
for every $w\in X$, there exists $v\in X$ such that
 $\mathcal D_fv=w$
and
 $\norm v_\infty\leq L'\norm w_\infty$,
 where $L':=4L$.
\end{corr}

\begin{proof}
As observed in Subsection~\ref{sec:induced-shadowing-results}, the cocycle $(f,Df)$
satisfies cocycle generalized hyperbolicity.  Apply
Theorem~\ref{thm:cocycle-D-surj} with $\alpha=f$ and $A=Df$.
\end{proof}

\subsection{Cocycle hyperbolicity over the perturbed base}
\label{sec:perturbed-cocycle}

\begin{thm}
\label{thm:g-Df-GH}
Assume that $f\in\DUC(\Bb)$ satisfies generalized hyperbolicity with
constants $C>0$ and $\lambda\in(0,1)$.  For every
$\lambda'\in(\lambda,1)$ there exist constants $C_1>0$ and $d_1>0$
such that the following holds.  If $g:\Bb\to\Bb$ is a homeomorphism
satisfying
\begin{equation}
\label{eq:g-f-cocycle-small}
 \sup_{x\in\Bb}\abs{g(x)-f(x)}<d_1,
\end{equation}
then the pair $(g,Df)$ satisfies cocycle generalized hyperbolicity with
constants $C_1,\lambda'$.
\end{thm}

\begin{proof}
We use the shadowing statements in Theorem~\ref{thm:GH-shadowing}
and the sequence robustness statement in
Lemma~\ref{lem:sequence-GH-robustness}.

Let \(R\) be as in \eqref{eq:C1-2}. Fix $\lambda'\in(\lambda,1)$ and take $\varepsilon,C_1>0$ from
Lemma~\ref{lem:sequence-GH-robustness} for $C,\lambda,\lambda',R$.
Let $M,d_0>0$ be the constants from Theorem~\ref{thm:GH-shadowing}.
By \eqref{eq:C1-r2}, we may choose $d_1\in(0,d_0)$ so small that
\begin{equation}
\label{eq:cocycle-r-small}
 r(Md_1)<\varepsilon.
\end{equation}

Let $g:\Bb\to\Bb$ be a homeomorphism satisfying
\eqref{eq:g-f-cocycle-small}.  Consider the
partition of $\Bb$ into orbits of $g$ and choose one point $y_0$ on each
orbit.  Put
$y_k:=g^k(y_0)$ for $k\in\Zz$.
Then
\[
 \abs{y_{k+1}-f(y_k)}
 =\abs{g(y_k)-f(y_k)}<d_1,
 \qquad k\in\Zz,
\]
so $\{y_k\}$ is a $d_1$-pseudotrajectory of $f$.
Lipschitz shadowing gives an exact trajectory $\{x_k\}_{k\in\Zz}$
of $f$ such that
\begin{equation}
\label{eq:cocycle-shadowing}
 \abs{x_k-y_k}\leq Md_1,
 \qquad k\in\Zz.
\end{equation}
If the orbit of $y_0$ has least period $m$, use periodic shadowing to
choose $\{x_k\}$ to be $m$-periodic.

Set
$
 A_k:=Df(x_k)$,
 $B_k:=Df(y_k)$.
The sequence $\{A_k\}$ satisfies generalized hyperbolicity with constants
$C,\lambda$, using the spaces $E_{x_k}^s,E_{x_k}^u$.  On the other hand,
\eqref{eq:C1-r1}, \eqref{eq:cocycle-shadowing}, and
\eqref{eq:cocycle-r-small} give
\[
 \norm{A_k-B_k}
 \leq r(\abs{x_k-y_k})
 \leq r(Md_1)<\varepsilon.
\]
Lemma~\ref{lem:sequence-GH-robustness}~\ref{item:sequence-robustness-general}
therefore supplies projections
$\widetilde P_k,\widetilde Q_k$ and a generalized-hyperbolic splitting
$
 \Bb=\widetilde E_k^s\oplus\widetilde E_k^u
$
for $\{B_k\}$, with constants $C_1,\lambda'$.
If the orbit of $y_0$ has least period $m$, both sequences and the
original splitting are $m$-periodic, so the periodic robustness statement
in Lemma~\ref{lem:sequence-GH-robustness}~\ref{item:sequence-robustness-periodic}
allows us to choose the new projections and splitting to be $m$-periodic as well.

Define
 $\widetilde E_{y_k}^{s,u}:=\widetilde E_k^{s,u}$,
 $\widetilde P_{y_k}:=\widetilde P_k$,
 $\widetilde Q_{y_k}:=\widetilde Q_k$,
This is well-defined: 
on an orbit of least period $m$ the chosen
projections and splitting are $m$-periodic.

Making these choices for every orbit of $g$, we obtain a splitting at
every point of $\Bb$.  The uniform projection bounds, inclusions, and
exponential estimates supplied by sequence robustness are independent of
the orbit and are precisely conditions \textup{(C-GH1)}--\textup{(C-GH3)}
for the cocycle $(g,Df)$.  This completes the proof.
\end{proof}

\begin{remark}
Only the forward distance in \eqref{eq:g-f-cocycle-small} is used in the
proof; no estimate of $g^{-1}-f^{-1}$ is required.  The assumption that
$g$ is a homeomorphism is used only to obtain the bi-infinite orbit
decomposition required by the cocycle construction.  The orbitwise
construction uses the Axiom of Choice.  
\end{remark}

\begin{corr}
\label{cor:Dg-surj}
Assume that $f\in\DUC(\Bb)$ satisfies generalized hyperbolicity.  There
exist constants $K_1>0$ and $d_c>0$ such that the following holds.  If
$g:\Bb\to\Bb$ is a homeomorphism and
$
 \norm{g-f}_\infty<d_c,
$
then, for every $w\in X$, there exists $v\in X$ satisfying
\begin{equation*}
 \mathcal D_{g,Df}v=w,
 \qquad
 \norm v_\infty\leq K_1\norm w_\infty,
\end{equation*}
where
\begin{equation*}
 (\mathcal D_{g,Df}v)(x):=v(g(x))-Df(x)v(x).
\end{equation*}
The constants are independent of $g$.
\end{corr}

\begin{proof}
Fix $\lambda'\in(\lambda,1)$, where $\lambda$ is an exponent for the
generalized hyperbolicity of $f$.  Theorem~\ref{thm:g-Df-GH} gives
$C_1,d_c>0$ such that $(g,Df)$ satisfies cocycle generalized
hyperbolicity with constants $C_1,\lambda'$.  Apply
Theorem~\ref{thm:cocycle-D-surj} to this cocycle.  Its estimate is
uniform because $K_1=K_c(C_1,\lambda')$ depends only on the displayed
cocycle constants.
\end{proof}

\section{Semiconjugacies}
\label{sec:semiconjugacies}
\label{sec:nonlinear}

\begin{proof}[Proof of Theorem~\ref{thm:topological-stability}]
Let $K_1,d_c$ be the constants from Corollary~\ref{cor:Dg-surj}.
Choose a radius $\rho>0$ so small that
$K_1r(\rho)\leq\frac12$; this is possible by \eqref{eq:C1-r2}.  Set
$
 M:=2K_1$,
 $d_0:=\min\left\{d_c,\frac{\rho}{2K_1}\right\}$.
Let $g:\Bb\to\Bb$ be a homeomorphism with
$\sup_{x\in\Bb}\abs{g(x)-f(x)}\leq d$, where $0<d<d_0$, and put
$\Delta:=g-f$.  Then $\Delta$ is continuous and
$\norm\Delta_\infty\leq d$, so $\Delta\in X$; no Lipschitz property of
$\Delta$ will be used.  Since $g$ is a homeomorphism and
$\norm{g-f}_\infty\leq d<d_c$, Corollary~\ref{cor:Dg-surj} applies.

Writing the identity $f\circ(\Id+k)=(\Id+k)\circ g$ at a point
$x\in\Bb$ and using \eqref{eq:sdef}, we obtain
\[
 f(x)+Df(x)k(x)+s_f(x,k(x))
 =f(x)+\Delta(x)+k(g(x)).
\]
Consequently, the required identity is equivalent to
\begin{equation}
\label{eq:reverse-cohom}
 \mathcal D_{g,Df} k=\mathcal N(k):=s_f(\cdot,k)-\Delta.
\end{equation}
For $k\in X$ with
$\norm k_\infty\leq\rho$, the field $\mathcal N(k)$ is continuous and
bounded, since $\abs{s_f(x,v)}\leq\abs v\,r(\abs v)$; thus
$\mathcal N$ maps the closed ball of radius $\rho$ in $X$ into $X$, and
$\mathcal N(0)=-\Delta$.  For $k,\widetilde k$ in this ball, the
perturbation cancels:
\[
 \mathcal N(k)(x)-\mathcal N(\widetilde k)(x)
 =s_f(x,k(x))-s_f(x,\widetilde k(x)),
 \qquad x\in\Bb,
\]
and \eqref{eq:s-rho} gives
$
 \norm{\mathcal N(k)-\mathcal N(\widetilde k)}_\infty
 \leq r(\rho)\norm{k-\widetilde k}_\infty.
$
Thus $\mathcal N$ is Lipschitz with constant $r(\rho)$, and no
regularity of $\Delta$ is needed.
Corollary~\ref{cor:Dg-surj} verifies
\eqref{eq:abstract-preimage} for $\mathcal L=\mathcal D_{g,Df}$, and
\[
 \frac{K_1\norm{\Delta}_\infty}{1-K_1r(\rho)}
 \leq2K_1d=Md<\rho.
\]
Lemma~\ref{lem:successive-corrections}, applied to
\eqref{eq:reverse-cohom}, therefore gives a solution $k\in X$ with
$\norm k_\infty\leq Md$.
By \eqref{eq:reverse-cohom}, $K=\Id+k$ satisfies
$f\circ K=K\circ g$, which proves \eqref{eq:top-stability-semiconjugacy}.
\end{proof}

\begin{proof}[Proof of Theorem~\ref{thm:main}]
Let $\Delta:\Bb\to\Bb$ be a bounded Lipschitz map with
\begin{equation}
\label{eq:Delta-small}
 \norm{\Delta}_\infty\leq d,
 \qquad
 \Lip(\Delta)\leq d,
\end{equation}
where $d>0$ will be restricted below, and put
$g:=f+\Delta$.

\medskip
\noindent\textbf{Forward semiconjugacy.}
Write $H=\Id+h$.  The required identity $g\circ H=H\circ f$ is
\begin{equation}
\label{eq:expanded-semi}
 f(x+h(x))+\Delta(x+h(x))=f(x)+h(f(x)).
\end{equation}
Using \eqref{eq:sdef}, equation \eqref{eq:expanded-semi} is equivalent to
\begin{equation}
\label{eq:nonlinear-cohom}
 \mathcal D_fh
 =s_f(\cdot,h)+\Delta\circ(\Id+h).
\end{equation}

Let $L'=4L$ be the constant from Corollary~\ref{prop:D-surj}.
Choose $\rho>0$ so small that
$L' r(\rho)\leq\frac14$.
This is possible by \eqref{eq:C1-r2}.  Put
$
 M_1:=2L'=8L$,
$d_1:=\min\left\{\frac{1}{4L'},\frac{\rho}{2L'}\right\}$.
Fix $d\in(0,d_1)$ and a perturbation satisfying
\eqref{eq:Delta-small}.  Then
\begin{equation}
\label{eq:theta}
 \theta:=L'\bigl(r(\rho)+d\bigr)\leq\frac12,
 \qquad
 M_1d=2L'd<\rho.
\end{equation}

By \eqref{eq:s-rho} and the Lipschitz estimate for $\Delta$, the
right-hand side of \eqref{eq:nonlinear-cohom} defines a map from the
closed ball of radius $\rho$ in $X$ into $X$, with Lipschitz constant
at most $r(\rho)+d$ and value $\Delta$ at $h=0$.
Corollary~\ref{prop:D-surj} verifies
\eqref{eq:abstract-preimage} for $\mathcal L=\mathcal D_f$, while
\eqref{eq:theta} gives
\[
 \frac{L'\norm{\Delta}_\infty}{1-\theta}
 \leq2L'd=M_1d<\rho.
\]
Lemma~\ref{lem:successive-corrections}, applied to
\eqref{eq:nonlinear-cohom}, therefore gives a solution $h\in X$ with
\begin{equation}
\label{eq:hbound-final}
 \norm h_\infty\leq M_1d.
\end{equation}
By \eqref{eq:nonlinear-cohom}, $H=\Id+h$ satisfies
$g\circ H=H\circ f$, with the required estimate
\eqref{eq:hbound-final}.

\medskip
\noindent\textbf{The map $g$ is a bi-Lipschitz homeomorphism.}
Let $R$ be as in \eqref{eq:C1-2}, and require $d<\frac{1}{2R}$.
Put
$\zeta:=\Delta\circ f^{-1}$ and $J:=\Id+\zeta$.
The mean-value estimate and \eqref{eq:C1-2} give
$
 \Lip(\zeta)\leq R\Lip(\Delta)\leq Rd<\frac12.
$
Hence
$J$ is surjective.
Moreover,
\[
 \abs{J(y)-J(\widetilde y)}
 \geq (1-Rd)\abs{y-\widetilde y},
 \qquad y,\widetilde y\in\Bb.
\]
Thus $J$ is a bi-Lipschitz homeomorphism.  Since
$g=J\circ f$,
$g$ is also a bi-Lipschitz homeomorphism.

\medskip
\noindent\textbf{Reverse semiconjugacy.}
Let $M_2,d_2$ denote the constants $M,d_0$ of
Theorem~\ref{thm:topological-stability}, and require $d<d_2$.  Since
$g$ is a homeomorphism with
$\sup_{x\in\Bb}\abs{g(x)-f(x)}=\norm{\Delta}_\infty\leq d$,
Theorem~\ref{thm:topological-stability} gives $k\in X$ with
$f\circ(\Id+k)=(\Id+k)\circ g$ and $\norm k_\infty\leq M_2d$, which is
the second identity in \eqref{eq:main-semiconjugacy-reverse}.
Finally, take $M:=\max\{M_1,M_2\}$ and
$d_0:=\min\left\{d_1,d_2,\frac{1}{2R}\right\}$.
For every $0<d<d_0$, the two constructions give the required
semiconjugacies with $\norm h_\infty,\norm k_\infty\leq Md$.
\end{proof}

\section{Infinite Product of Morse--Smale Systems}
\label{sec:infinite-morse-smale}

We prove structural stability of the coordinatewise Morse--Smale product
on $\ell^p(\Zz)$, $1\leq p\leq+\infty$.
Theorem~\ref{thm:main} gives semiconjugacies in both directions, but does
not establish that they are mutually inverse.  Here we construct
$H=\Id+h$ and $K=\Id+k$ with compatible normalizations and prove that
the two composition defects
\[
 \xi_1=K\circ H-\Id=h+k\circ H,\qquad
 \xi_2=H\circ K-\Id=k+h\circ K
\]
vanish.

\subsection{The product map and elementary estimates}

Throughout this section, $a$, $c$, $\lambda_1$, $M$, $\Bb$, and $f$
are as in Subsection~\ref{sec:product-results}, with a fixed
$p\in[1,+\infty]$.
The inverse is $f^{-1}(x)_j=a^{-1}(x_j)$, and the
derivative is the diagonal operator
$
 Df(x)\in\mathcal L(\Bb)$, $(Df(x)v)_j=a'(x_j)v_j.
$
Moreover,
\begin{equation*}
 \|Df(x)\|\leq\lambda_1^{-1},
 \qquad
 \|Df(x)^{-1}\|\leq\lambda_1^{-1},
 \qquad x\in\Bb.
\end{equation*}
Because the norm of a diagonal multiplier on $\ell^p$ is the supremum of
the absolute values of its diagonal coefficients, \eqref{eq:product-a-bounds}
gives
\begin{equation}
\label{eq:product-Df-Lip}
 \|Df(x)-Df(y)\|
 =\sup_{j\in\Zz}|a'(x_j)-a'(y_j)|
 \leq M|x-y|_\infty
 \leq M|x-y|_p.
\end{equation}

We use the following estimate for the nonlinear remainder $s_f$ from \eqref{eq:sdef}:
\begin{equation*}
 |s_f(x,v)|_p
 \leq \frac M2 |v|_\infty |v|_p
 \leq \frac M2 |v|_p^2.
\end{equation*}
If $|u|_p,|v|_p\leq\rho$, then the derivative of the map
$v\mapsto s_f(x,v)$ is $Df(x+v)-Df(x)$.  Consequently,
\begin{equation}
\label{eq:product-s-Lip-v}
 |s_f(x,u)-s_f(x,v)|_p\leq M\rho|u-v|_p.
\end{equation}
For fixed $v$ with $|v|_p\leq\rho$, the coordinatewise mean-value theorem gives
\begin{equation}
\label{eq:product-s-Lip-x}
 |s_f(x,v)-s_f(y,v)|_p
 \leq 2M|v|_\infty|x-y|_p
 \leq 2M\rho|x-y|_p.
\end{equation}
Indeed, the scalar mean-value theorem and $|a''|\leq M$ give, for every
coordinate $j$,
\[
 |(s_f(x,v)-s_f(y,v))_j|\leq 2M|v_j|\,|x_j-y_j|.
\]
Taking the $\ell^p$ norm yields \eqref{eq:product-s-Lip-x}, including
when $p=\infty$.

\subsection{One-dimensional estimates}

We collect the scalar pseudotrajectory estimates that control the
coefficients of the cohomological operators below.  The three lemmas
separate the cutoff and trapping geometry, the exponential bounds on
subsegments, and the bounds on Green coefficients and their Lipschitz
dependence.  Empty products are understood as one.

\begin{lem}
\label{lem:product-scalar-trapping}
There exist a Lipschitz function $\eta:\mathbb R\to[0,1]$, positive
numbers $\gamma,\beta$, and $\delta_*>0$, depending only on the scalar
diffeomorphism $a$, with the following properties.  Let
$(t_n)_{n\in\Zz}$ be any scalar pseudotrajectory satisfying
\begin{equation}
\label{eq:product-scalar-pseudo}
 |t_{n+1}-a(t_n)|<\delta_*,
 \qquad n\in\Zz.
\end{equation}
\begin{enumerate}[label=\textup{(\roman*)},ref=(\roman*),leftmargin=10mm]
\item \label{item:product-scalar-cutoff}
\textbf{Cutoff.} For every $t\in\mathbb R$, the cutoff satisfies
\begin{equation}
\label{eq:product-eta-implications}
 \eta(t)\neq0\Longrightarrow |a(t)|>c+4\gamma,
 \qquad
 1-\eta(t)\neq0\Longrightarrow |t|<c-4\gamma.
\end{equation}

\item \label{item:product-scalar-margins}
\textbf{Margins.} For every $r\in\Zz$,
\begin{equation}
\label{eq:product-stable-margin}
 \eta(t_r)\neq0
 \quad\Longrightarrow\quad
 |t_k|>c+2\gamma+\beta
 \quad\text{for every }k\geq r+1.
\end{equation}
\begin{equation}
\label{eq:product-unstable-margin}
 1-\eta(t_r)\neq0
 \quad\Longrightarrow\quad
 |t_k|<c-2\gamma-\beta
 \quad\text{for every }k\leq r.
\end{equation}

\end{enumerate}
\end{lem}

\begin{lem}
\label{lem:product-scalar-pseudotrajectories}
Fix the cutoff and margins from Lemma~\ref{lem:product-scalar-trapping}.
After decreasing $\delta_*$ if necessary, there exist constants
$C_0>0$ and $\theta\in(0,1)$,
depending only on the map $a$, such that the following holds for every
$\delta_*$-pseudotrajectory.
For $q\in\Zz$ and $m\geq0$, a segment $t_q,\ldots,t_{q+m}$ contained
in $|t|>c+2\gamma$ satisfies
\begin{equation}
\label{eq:product-stable-subsegment}
 \left|\prod_{k=q}^{q+m-1}a'(t_k)\right|
 \leq C_0\theta^m,
\end{equation}
and a segment contained in $|t|<c-2\gamma$ satisfies
\begin{equation}
\label{eq:product-unstable-subsegment}
 \left|\prod_{k=q}^{q+m-1}\frac1{a'(t_k)}\right|
 \leq C_0\theta^m.
\end{equation}
\end{lem}

\begin{lem}
\label{lem:product-scalar-Green-Lipschitz}
Use the cutoff and constants from
Lemmas~\ref{lem:product-scalar-trapping}
and~\ref{lem:product-scalar-pseudotrajectories}.  For a $\delta_*$-pseudotrajectory $t=(t_k)_{k\in\Zz}$, define
\begin{equation*}
 b^s_{n,r}(t)=\eta(t_r)\prod_{k=r+1}^{n-1}a'(t_k) \qquad r<n,
\end{equation*}
\[
 b^u_{n,r}(t)=(1-\eta(t_r))\prod_{k=n}^{r}\frac1{a'(t_k)}, \qquad r\geq n.
\]
\begin{enumerate}[label=\textup{(\roman*)},ref=(\roman*),leftmargin=10mm]
\item \label{item:product-scalar-coefficients}
\textbf{Coefficient bounds.} After increasing $C_0$ if necessary,
every such pseudotrajectory satisfies
\begin{equation}
\label{eq:product-stable-coefficient}
 |b^s_{n,r}(t)|
 \leq C_0\theta^{n-r-1},
 \qquad r<n,
\end{equation}
\begin{equation}
\label{eq:product-unstable-coefficient}
 |b^u_{n,r}(t)|
 \leq C_0\theta^{r-n+1},
 \qquad n\leq r.
\end{equation}

\item \label{item:product-scalar-Green-Lipschitz}
\textbf{Lipschitz dependence.}
There exist $\rho_*>0$, $C_4>0$, and $\theta_1\in(\theta,1)$, depending
only on the map $a$, such that, whenever $t,\widetilde t$ are two $\delta_*$-pseudotrajectories and
$
 \varepsilon:=\sup_{k\in\Zz}|t_k-\widetilde t_k|<\rho_*,
$
one has
\begin{equation}
\label{eq:product-scalar-Green-Lip-s}
 |b^s_{n,r}(t)-b^s_{n,r}(\widetilde t)|
 \leq C_4\theta_1^{n-r-1}\varepsilon,
 \qquad r<n,
\end{equation}
\begin{equation}
\label{eq:product-scalar-Green-Lip-u}
 |b^u_{n,r}(t)-b^u_{n,r}(\widetilde t)|
 \leq C_4\theta_1^{r-n+1}\varepsilon,
 \qquad r\geq n.
\end{equation}
\end{enumerate}
\end{lem}

The proofs of all three lemmas are given in
Appendix~\ref{app:product-scalar-pseudotrajectories}.

\subsection{Green operators along corrected trajectories}
\label{sec:product-common-Green-operators}

Fix the cutoff and the final constants from
Lemmas~\ref{lem:product-scalar-trapping}-\ref{lem:product-scalar-Green-Lipschitz}.
For each $x=(x_j)_{j\in\Zz}\in\Bb$, define
\begin{equation*}
 E(x):\Bb\longrightarrow\Bb,
 \qquad
 E(x)v=(\eta(x_j)v_j)_{j\in\Zz}.
\end{equation*}
Then $E(x)\in\mathcal L(\Bb)$.  In this way we obtain a map
$E:\Bb\to\mathcal L(\Bb)$, $x\to E(x)$.
The map $E$ is generally nonlinear as a function of $x$, while for each
fixed $x$ the operator $E(x)$ is linear in $v$.  Moreover,
\begin{equation*}
 \|E(x)\|\leq1,
 \quad
 \|\Id_{\Bb}-E(x)\|\leq1,
 \quad
 \|E(x)-E(y)\|
 =\sup_{j\in\Zz}|\eta(x_j)-\eta(y_j)|
 \leq \operatorname{Lip}(\eta)|x-y|_p.
\end{equation*}

Let $C_4,\theta_1,\rho_*$ be the constants from
Lemma~\ref{lem:product-scalar-Green-Lipschitz}~\ref{item:product-scalar-Green-Lipschitz}.  Set
\begin{equation}
\label{eq:product-Green-sum-constants}
 L:=C_0\frac{1+\theta}{1-\theta},
 \qquad
 L_1:=C_4\frac{1+\theta_1}{1-\theta_1}.
\end{equation}
The constant $L$ sums the coefficient bounds for the two halves of the
Green formulas, while $L_1$ sums the bounds for coefficient differences.

Let $C_b^0(\Bb,\Bb)$ be the Banach space of bounded continuous maps
$w:\Bb\to\Bb$, with norm
$|w|_\infty:=\sup_{x\in\Bb}|w(x)|_p$.
Let $\alpha:\Bb\to\Bb$ be a homeomorphism; below we shall use
$\alpha=f$ and $\alpha=g$, where $g$ is a small perturbation of $f$.
For $u\in C_b^0(\Bb,\Bb)$, put
\begin{equation}
\label{eq:product-field-coefficients}
\begin{gathered}
 \Phi_u:=\Id+u,\qquad A_u:=Df\circ\Phi_u,\qquad E_u:=E\circ\Phi_u,\\
 \Phi_u(x)=x+u(x),\qquad A_u(x)=Df(x+u(x)),\qquad E_u(x)=E(x+u(x)).
\end{gathered}
\end{equation}
Call $u$ admissible for $\alpha$ if
\begin{equation}
\label{eq:product-Green-admissibility}
 \sup_{x\in\Bb}|\Phi_u(\alpha(x))-f(\Phi_u(x))|_p<\delta_*.
\end{equation}
Every coordinate of $\Phi_u(\alpha^r(x))$, $r\in\Zz$, then satisfies
\eqref{eq:product-scalar-pseudo}.  

For a fixed admissible $u$, our aim is to solve the linear cohomological
equation
\begin{equation}
\label{eq:product-Salpha-cohomological}
 v(\alpha(x))-Df(x+u(x))v(x)=w(x),\qquad x\in\Bb,
\end{equation}
for a bounded continuous field $v$.  
For $w\in C_b^0(\Bb,\Bb)$ and an admissible $u\in C_b^0(\Bb,\Bb)$, define
\begin{align}
\label{eq:product-Sg}
 (\mathcal S_\alpha(u)w)(x)
 &=\sum_{m=0}^{\infty}
 \bigl[A_u(\alpha^{-1}(x))\cdots A_u(\alpha^{-m}(x))\bigr]
 E_u(\alpha^{-m-1}(x))w(\alpha^{-m-1}(x))\notag\\
 &\quad-\sum_{m=0}^{\infty}
 \bigl[A_u(x)^{-1}\cdots A_u(\alpha^m(x))^{-1}\bigr]
 (\Id-E_u(\alpha^m(x)))w(\alpha^m(x)).
\end{align}
\begin{remark}
The operator $E$ plays a role similar to that of $\Xi_{1, 2}$ in \cite[Example 22]{TikhomirovGH}. A somewhat similar construction was used in
\cite{MessaoudiEtal} to study linear operators under the weaker assumption of
pseudo-hyperbolicity.
\end{remark}

Each $\mathcal S_\alpha(u)$ is linear in the forcing $w$;
its dependence on $u$ is generally nonlinear.  
In the bounds below, operator norms are taken on $C_b^0(\Bb,\Bb)$.

\begin{lem}[Properties of the Green operators]
\label{lem:product-Salpha}
Let $u\in C_b^0(\Bb,\Bb)$ be admissible for $\alpha$.
\begin{enumerate}[label=\textup{(\roman*)},ref=(\roman*),leftmargin=10mm]
\item\label{item:product-Salpha-bound}
\textbf{Boundedness.}  The series in \eqref{eq:product-Sg} converge
absolutely and uniformly in $x$, the operator $\mathcal S_\alpha(u)$
maps $C_b^0(\Bb,\Bb)$ into itself, and
\begin{equation}
\label{eq:product-Salpha-bound}
 \|\mathcal S_\alpha(u)\|\leq L.
\end{equation}
\item\label{item:product-Salpha-cohomological}
\textbf{Cohomological identity.}  For every $w\in C_b^0(\Bb,\Bb)$, the
field $v=\mathcal S_\alpha(u)w$ satisfies
\eqref{eq:product-Salpha-cohomological}; equivalently,
$\mathcal D_{\alpha,A_u}\mathcal S_\alpha(u)=\Id_{C_b^0(\Bb,\Bb)}$.
\item\label{item:product-Salpha-Lip}
\textbf{Lipschitz dependence on the displacement.}  If $\widetilde u$
is also admissible for $\alpha$ and $|u-\widetilde u|_\infty<\rho_*$,
then
\begin{equation}
\label{eq:product-Salpha-Lip}
 \|\mathcal S_\alpha(u)-\mathcal S_\alpha(\widetilde u)\|
 \leq L_1|u-\widetilde u|_\infty.
\end{equation}
\end{enumerate}
\end{lem}

\begin{proof}
\textit{(i)}
By Lemma \ref{lem:product-scalar-Green-Lipschitz} (i) 
the coefficients of
$w(\alpha^{-m-1}(x))$ and $w(\alpha^m(x))$ in
\eqref{eq:product-Sg} have norms at most $C_0\theta^m$ and
$C_0\theta^{m+1}$, respectively.  Hence the series converge absolutely
and uniformly in $x$, and \eqref{eq:product-Salpha-bound} follows.
Every finite partial sum is continuous; 
thus the uniform limit belongs to $C_b^0(\Bb,\Bb)$.

\textit{(ii)} Write $v=v^s+v^u$, where $v^s$ is the first sum in
\eqref{eq:product-Sg} and $v^u$ is the second sum with its minus sign.
The calculation is analogous to the reindexing of Perron sums.
Fix $x\in\Bb$ and abbreviate
\[
 A_k:=A_u(\alpha^k(x)),\qquad
 E_k:=E_u(\alpha^k(x)),\qquad
 w_k:=w(\alpha^k(x)),\qquad k\in\Zz.
\]
All products below retain the displayed order, and empty products
equal $\Id$.  Absolute convergence allows us to separate the term
$m=0$ and put $n=m-1$ in the remaining sum.  For the stable part,
\begin{align*}
 v^s(\alpha(x))
 &=E_0w_0+\sum_{m=1}^{\infty}
 [A_0\cdots A_{1-m}]E_{-m}w_{-m}\\
 &=E_0w_0+A_0\sum_{n=0}^{\infty}
 [A_{-1}\cdots A_{-n}]E_{-n-1}w_{-n-1}\\
 &=E_0w_0+A_0v^s(x).
\end{align*}
For the unstable part, multiplication by $A_0$ cancels the first
inverse factor in every term, so the same reindexing gives
\begin{align*}
 A_0v^u(x)
 &=-\sum_{m=0}^{\infty}
 [A_1^{-1}\cdots A_m^{-1}](\Id-E_m)w_m\\
 &=-(\Id-E_0)w_0-\sum_{n=0}^{\infty}
 [A_1^{-1}\cdots A_{n+1}^{-1}](\Id-E_{n+1})w_{n+1}\\
 &=-(\Id-E_0)w_0+v^u(\alpha(x)).
\end{align*}
Thus, in the original notation,
\[
 v^s(\alpha(x))-A_u(x)v^s(x)=E_u(x)w(x),\qquad
 v^u(\alpha(x))-A_u(x)v^u(x)=(\Id-E_u(x))w(x).
\]
Adding these identities proves \eqref{eq:product-Salpha-cohomological}.

\textit{(iii)} Put $\varepsilon:=|u-\widetilde u|_\infty$.  Fix $x\in\Bb$
and $j\in\Zz$, and set
\[
 t_r=(\Phi_u(\alpha^r(x)))_j,\qquad
 \widetilde t_r=(\Phi_{\widetilde u}(\alpha^r(x)))_j,\qquad r\in\Zz.
\]
Admissibility implies that $t=(t_r)$ and $\widetilde t=(\widetilde t_r)$
are scalar $\delta_*$-pseudotrajectories.  Moreover,
\[
 |t_r-\widetilde t_r|
 =|(u(\alpha^r(x))-\widetilde u(\alpha^r(x)))_j|
 \leq\varepsilon<\rho_*.
\]
Applying \eqref{eq:product-scalar-Green-Lip-s} with $(n,r)=(0,-m-1)$
and \eqref{eq:product-scalar-Green-Lip-u} with $(n,r)=(0,m)$ therefore
gives, for every $m\geq0$,
\begin{equation*}
 |b^s_{0,-m-1}(t)-b^s_{0,-m-1}(\widetilde t)|
 \leq C_4\theta_1^m\varepsilon, \qquad
 |b^u_{0,m}(t)-b^u_{0,m}(\widetilde t)|
 \leq C_4\theta_1^{m+1}\varepsilon.
\end{equation*}
The coefficients in \eqref{eq:product-Sg} are diagonal operators with
these scalar coefficients, so taking the supremum over $j$ gives the
same bounds for their differences in operator norm, for every
$p\in[1,+\infty]$.  Since the forcing values
$w(\alpha^{-m-1}(x))$ and $w(\alpha^m(x))$ do not depend on $u$,
\eqref{eq:product-Sg} now yields
\begin{equation*}
 |((\mathcal S_\alpha(u)-\mathcal S_\alpha(\widetilde u))w)(x)|_p
 \leq C_4\varepsilon |w|_\infty
 \sum_{m=0}^{\infty}(\theta_1^m+\theta_1^{m+1})
 =C_4\frac{1+\theta_1}{1-\theta_1}\varepsilon |w|_\infty
 =L_1\varepsilon |w|_\infty.
\end{equation*}
Taking the supremum over $x$ and then over $|w|_\infty\leq1$ proves
\eqref{eq:product-Salpha-Lip}, with $L_1$ from \eqref{eq:product-Green-sum-constants}.
\end{proof}

\textbf{Covariance under composition.} The following statement is critical for the subsequent pullback compatibility argument and motivated the precise form \eqref{eq:product-Sg}.
For a continuous map $\psi:\Bb\to\Bb$, let
\begin{equation*}
 C_{\Id+\psi}:C_b^0(\Bb,\Bb)\longrightarrow C_b^0(\Bb,\Bb),
 \qquad C_{\Id+\psi} w:=w\circ(\Id+\psi).
\end{equation*}
This is a bounded linear operator with $\|C_{\Id+\psi}\|\leq1$.

\begin{lem}
\label{lem:product-Green-covariance}
Let $\alpha,\beta\in\{f,g\}$ and $\psi\in C_b^0(\Bb,\Bb)$ satisfy
$\alpha\circ(\Id+\psi)=(\Id+\psi)\circ\beta$.
For a field $u$ admissible for $\alpha$, put
$
 v=\psi+u\circ(\Id+\psi).
$
Then $\Phi_v=\Phi_u\circ(\Id+\psi)$, the field $v$ is admissible
for $\beta$, and
\begin{equation}
\label{eq:product-Green-covariance}
 C_{\Id+\psi}\mathcal S_\alpha(u)
 =\mathcal S_\beta(v)C_{\Id+\psi}.
\end{equation}
No invertibility of $\Id+\psi$ is assumed.
\end{lem}

\begin{proof}
The definition of $v$ gives the identity of corrected points
$
 (\Id+u)\circ(\Id+\psi)=\Id+v.
$
Thus $\Phi_u\circ(\Id+\psi)=\Phi_v$ and, in particular,
\begin{align*}
 Df\circ(\Id+u)\circ(\Id+\psi)&=Df\circ(\Id+v),\\
 E\circ(\Id+u)\circ(\Id+\psi)&=E\circ(\Id+v).
\end{align*}
In other words, precomposition by $\Id+\psi$ transforms the coefficient
maps $A_u,E_u$ into $A_v,E_v$.

The assumed semiconjugacy relation gives
\[
 \Phi_v(\beta(x))-f(\Phi_v(x))
 =\Phi_u(\alpha(x+\psi(x)))-f(\Phi_u(x+\psi(x))).
\]
Taking the supremum over $x$ and using
\eqref{eq:product-Green-admissibility} for $u$ proves the admissibility of $v$
for $\beta$.

Since $\alpha,\beta$ are invertible,
$\alpha^r\circ(\Id+\psi)=(\Id+\psi)\circ\beta^r$ for every $r\in\Zz$.
Together with the identity of corrected points, this yields
$
 \Phi_u\circ\alpha^r\circ(\Id+\psi)=\Phi_v\circ\beta^r.
$
For $w\in C_b^0(\Bb,\Bb)$, the derivative and cutoff factors in the Green formula
\eqref{eq:product-Sg} for $(\mathcal S_\alpha(u)w)(x+\psi(x))$
are exactly those for $(\mathcal S_\beta(v)C_{\Id+\psi}w)(x)$,
with the same products and inverses.  The forcing terms agree as well,
since
$
 w\circ\alpha^r\circ(\Id+\psi)
 =(C_{\Id+\psi}w)\circ\beta^r.
$
The two absolutely convergent series therefore agree term by term,
which proves \eqref{eq:product-Green-covariance}.
\end{proof}

\subsection{Linear and nonlinear cohomological equations}
\label{sec:product-cohomological-operators}

Throughout the rest of this section, let $\Delta:\Bb\to\Bb$ be a
bounded Lipschitz map and put
$g=f+\Delta$.
Assume that, for some $d>0$,
\begin{equation}
\label{eq:product-Delta-estimates}
 |\Delta|_\infty\leq d,
 \qquad
 |\Delta(x)-\Delta(y)|_p\leq d|x-y|_p.
\end{equation}
For $d<\lambda_1/2$, as required in
\eqref{eq:product-smallness-list} below, the map $g$ is a bi-Lipschitz
homeomorphism.  Indeed, $\Lip(f^{-1})\leq\lambda_1^{-1}$, so the
argument of Section~\ref{sec:semiconjugacies}, applied to
$g=(\Id+\Delta\circ f^{-1})\circ f$, gives
\[
 \Lip(g^{-1})\leq\frac{\lambda_1^{-1}}{1-d/\lambda_1}
 = (\lambda_1-d)^{-1}.
\]
Writing $H=\Id+h$ and $K=\Id+k$, the semiconjugacy identities
$g\circ H=H\circ f$ and $K\circ g=f\circ K$ take the form
\[
\begin{aligned}
 h(f(x))-Df(x)h(x)&=G(h)(x),\\
 k(g(y))-Df(K(y))k(y)&=R(k)(y).
\end{aligned}
\]
For bounded continuous fields $h,k$, define the remainder operators by
\begin{equation}
\label{eq:product-G-def}
 G(h)(x):=s_f(x,h(x))+\Delta(x+h(x)),
\end{equation}
\begin{equation*}
 R(k)(y):=f(y+k(y))-f(y)-Df(y+k(y))k(y)-\Delta(y).
\end{equation*}

The reverse remainder can also be written as
\[
 R(k)(y)=-s_f(K(y),-k(y))-\Delta(y).
\]
The estimates for $G$ and $R$ are proved in Subsection~\ref{sec:product-remainders}.

Recall that
$(\mathcal D_fv)(x)=v(f(x))-Df(x)v(x)$.
Since $0$ is admissible for $f$, define
\begin{equation}
\label{eq:product-Of}
 O_f:=\mathcal S_f(0).
\end{equation}

\begin{lem}
\label{lem:product-Of}
The operator $O_f:C_b^0(\Bb,\Bb)\to C_b^0(\Bb,\Bb)$ is bounded and
linear and satisfies
\begin{equation}
\label{eq:product-Of-bound}
 |O_fw|_\infty\leq L|w|_\infty, \qquad \mathcal D_fO_f=\Id_{C_b^0(\Bb,\Bb)}.
\end{equation}
\end{lem}

\begin{proof}
Apply Lemma~\ref{lem:product-Salpha} with $\alpha=f$ and $u=0$:
items \ref{item:product-Salpha-bound} and
\ref{item:product-Salpha-cohomological} give
\eqref{eq:product-Of-bound}.
\end{proof}

Use $\delta_*$ from Lemma~\ref{lem:product-scalar-pseudotrajectories},
$\rho_*$ from Lemma~\ref{lem:product-scalar-Green-Lipschitz}~\ref{item:product-scalar-Green-Lipschitz}, and
$L,L_1$ from \eqref{eq:product-Green-sum-constants}.
Fix $\rho>0$ and put
$\varepsilon:=d+\frac M2\rho^2$.
We impose the same smallness conditions on these constants throughout:
\begin{equation}
\label{eq:product-smallness-list}
\begin{aligned}
 d+(1+\lambda_1^{-1})\rho<\delta_*, \quad
 &2\rho<\rho_*,
 d<\lambda_1/2,\\
L \varepsilon = L\left(d+\frac M2\rho^2\right)\leq\rho, \quad
 &L(M\rho+d)<\frac12,\\
 LM\rho+L_1\left(d+\frac M2\rho^2\right)<\frac12, \quad  &L_1 \varepsilon < \frac12.
\end{aligned}
\end{equation}

The first two inequalities ensure admissibility of corrected trajectories
and allow comparison of Green operators.  The bound $L\varepsilon\leq\rho$
keeps the fixed-point maps in the ball of radius $\rho$, while the
contraction bounds control convergence and force the composition defects
to vanish.  In Subsection~\ref{sec:product-structural-stability}, we will
take $\rho=2Ld$ and $d$ sufficiently small to satisfy all these conditions
and use them to construct mutually inverse semiconjugacies.

Set
\begin{equation}
\label{eq:product-forward-ball-space}
 \mathcal U_\rho:=\{u\in C_b^0(\Bb,\Bb):|u|_\infty\leq\rho\},\qquad
 \mathcal W_\varepsilon:=\{w\in C_b^0(\Bb,\Bb):|w|_\infty\leq\varepsilon\}.
\end{equation}
We shall use $G$ and $R$ restricted to $\mathcal U_\rho$.
For $u\in\mathcal U_\rho$, we have
\begin{equation}
\label{eq:product-field-defect}
 |\Phi_u(g(y))-f(\Phi_u(y))|_p
 \leq d+(1+\lambda_1^{-1})\rho<\delta_*
\end{equation}
Thus every $u\in\mathcal U_\rho$ is admissible for $g$, and
$\mathcal S_g(u)$ is defined by \eqref{eq:product-Sg} with $\alpha=g$.
For every $u\in\mathcal U_\rho$ and $w\in C_b^0(\Bb,\Bb)$,
the field $v=\mathcal S_g(u)w$ satisfies the linear cohomological equation
(Lemma~\ref{lem:product-Salpha}\ref{item:product-Salpha-cohomological})
\begin{equation*}
 v(g(y))-Df(y+u(y))v(y)=w(y),
 \qquad y\in\Bb.
\end{equation*}

The reverse construction uses two fixed-point problems.  For each
prescribed forcing $w\in\mathcal W_\varepsilon$, the next lemma defines
$O_gw$ as the unique fixed point of $u\mapsto\mathcal S_g(u)w$ in
$\mathcal U_\rho$.  In Subsection~\ref{sec:product-structural-stability},
we then solve $k=O_gR(k)$, where the forcing depends on the unknown $k$,
to obtain the reverse semiconjugacy.

\begin{lem}
\label{lem:product-Og}
Assume \eqref{eq:product-Delta-estimates} and
\eqref{eq:product-smallness-list}.
Define the local nonlinear solution operator
$O_g:\mathcal W_\varepsilon\longrightarrow\mathcal U_\rho$
by the normalized implicit equation
\begin{equation}
\label{eq:product-Og-definition}
 u=O_gw\quad\Longleftrightarrow\quad
 u\in\mathcal U_\rho\ \hbox{ and }\ u=\mathcal S_g(u)w.
\end{equation}
The following assertions hold.
\begin{enumerate}[label=\textup{(\roman*)},ref=(\roman*),leftmargin=10mm]
\item\label{item:product-Og-existence}
For every $w\in\mathcal W_\varepsilon$, equation
\eqref{eq:product-Og-definition} has a unique solution
$u\in\mathcal U_\rho$.  It satisfies
\begin{equation}
\label{eq:product-Og-bound}
 |O_gw|_\infty\leq L|w|_\infty.
\end{equation}
\item\label{item:product-Og-Lipschitz}
Put
$L_g:=\frac{L}{1-L_1\varepsilon}$.
Then
\begin{equation}
\label{eq:product-Og-Lip}
 |O_gw-O_g\widetilde w|_\infty
 \leq L_g|w-\widetilde w|_\infty,
 \qquad w,\widetilde w\in\mathcal W_\varepsilon.
\end{equation}
\item\label{item:product-Og-normalization}
The normalized solution $u=O_gw$ solves the nonlinear cohomological
equation
\begin{equation}
\label{eq:product-Og-cohomological}
 u(g(y))-Df(y+u(y))u(y)=w(y).
\end{equation}
\end{enumerate}
\end{lem}
\begin{remark}
Uniqueness in item~\ref{item:product-Og-existence} refers to the
normalized equation \eqref{eq:product-Og-definition}, not to all small
solutions of~\eqref{eq:product-Og-cohomological}.
\end{remark}      

\begin{proof}
Fix $w\in\mathcal W_\varepsilon$.  By
\eqref{eq:product-field-defect}, every $u\in\mathcal U_\rho$ is
admissible for $g$.  Thus $\mathcal S_g(u)w$ is defined, and
\eqref{eq:product-Salpha-bound} gives
\[
 |\mathcal S_g(u)w|_\infty
 \leq\|\mathcal S_g(u)\|\,|w|_\infty
 \leq L\varepsilon\leq\rho.
\]
Hence $u\mapsto\mathcal S_g(u)w$ maps $\mathcal U_\rho$ into itself.
For $u,\widetilde u\in\mathcal U_\rho$, we also have
$
 |u-\widetilde u|_\infty
 \leq 2\rho<\rho_*.
$
Consequently, \eqref{eq:product-Salpha-Lip} applies and yields
\begin{equation*}
 |\mathcal S_g(u)w-\mathcal S_g(\widetilde u)w|_\infty
 \leq\|\mathcal S_g(u)-\mathcal S_g(\widetilde u)\|\,|w|_\infty
 \leq L_1|u-\widetilde u|_\infty\,|w|_\infty
 \leq L_1\varepsilon|u-\widetilde u|_\infty.
\end{equation*}

Since $L_1\varepsilon<1$ and $\mathcal U_\rho$ is a closed ball in
the Banach space $C_b^0(\Bb,\Bb)$, the contraction mapping theorem
gives a unique fixed point $O_gw\in\mathcal U_\rho$.  At this fixed
point, \eqref{eq:product-Salpha-bound} gives
\[
 |O_gw|_\infty
 =|\mathcal S_g(O_gw)w|_\infty
 \leq L|w|_\infty,
\]
which proves \eqref{eq:product-Og-bound}.

Now let $w,\widetilde w\in\mathcal W_\varepsilon$, and put
$u=O_gw$ and $\widetilde u=O_g\widetilde w$.
Using their fixed-point equations and the linearity of
$\mathcal S_g(u)$ in the forcing, we obtain
\begin{equation*}
 u-\widetilde u
 =\mathcal S_g(u)w-\mathcal S_g(\widetilde u)\widetilde w
 =\mathcal S_g(u)(w-\widetilde w)
   +(\mathcal S_g(u)-\mathcal S_g(\widetilde u))\widetilde w.
\end{equation*}
Applying \eqref{eq:product-Salpha-bound} to the first term and
\eqref{eq:product-Salpha-Lip} to the second gives
\begin{equation*}
 |u-\widetilde u|_\infty
 \leq L|w-\widetilde w|_\infty
       +L_1|u-\widetilde u|_\infty\,|\widetilde w|_\infty
 \leq L|w-\widetilde w|_\infty
       +L_1\varepsilon|u-\widetilde u|_\infty.
\end{equation*}
Moving the last term to the left and dividing by
$1-L_1\varepsilon>0$, we conclude that
\[
 |O_gw-O_g\widetilde w|_\infty
 \leq\frac{L}{1-L_1\varepsilon}|w-\widetilde w|_\infty
 =L_g|w-\widetilde w|_\infty.
\]
This proves \eqref{eq:product-Og-Lip}.

Finally, apply \eqref{eq:product-Salpha-cohomological} with
$\alpha=g$ at the fixed point $u=\mathcal S_g(u)w=O_gw$.
This gives \eqref{eq:product-Og-cohomological}.
\end{proof}

\begin{remark}
\label{rem:product-alternative-mechanism}
The normalization \eqref{eq:product-Og-definition} provides an
alternative to the structural-stability mechanism in
\cite[Theorem~1]{BernardesMessaoudi} and
\cite[Theorem~16\textup{(S2)}]{TikhomirovGH}.
There, the displacements of both semiconjugacies take values in a
fixed closed subspace, so the displacements of their compositions
remain in the same normalization class and uniqueness yields the
inverse identities.  Here the normalization depends on the unknown.
The inverse identities
will follow from the pullback compatibility and comparison estimates
of Lemma~\ref{lem:product-pullback}.

Normalized Green constructions also appear in the theorem of
Backes and Dragi\v{c}evi\'c \cite{BackesDragicevicGrobmanHartman}
for nonautonomous dynamics with generalized hyperbolicity; see also
\cite{BackesDragicevicPalmer} for a related inverse argument.
\end{remark}

\subsection{Pullback compatibility of the cohomological operators}
\label{sec:product-pullback}

\begin{lem}[Pullback compatibility of the cohomological operators]
\label{lem:product-pullback}
Under \eqref{eq:product-Delta-estimates} and
\eqref{eq:product-smallness-list}, let
$w_f,w_g\in\mathcal W_\varepsilon$ and put
\begin{equation*}
 h=O_fw_f,\qquad k=O_gw_g,\qquad
 H=\Id+h,\qquad K=\Id+k.
\end{equation*}
Then $h,k\in\mathcal U_\rho$, and the following assertions hold.
\begin{enumerate}[label=\textup{(\roman*)},ref=(\roman*),leftmargin=10mm]
\item\label{item:product-pullback-forward}
If $g\circ H=H\circ f$, then
\begin{equation}
\label{eq:product-pullback-forward}
 |h+k\circ H|_\infty
 \leq L_g|w_f+w_g\circ H|_\infty.
\end{equation}
\item\label{item:product-pullback-reverse}
If $K\circ g=f\circ K$, then
\begin{equation}
\label{eq:product-pullback-reverse}
 |k+h\circ K|_\infty
 \leq L|w_g+w_f\circ K|_\infty.
\end{equation}
\end{enumerate}
\end{lem}

\begin{proof}
The bounds in Lemmas~\ref{lem:product-Of} and \ref{lem:product-Og},
together with $L\varepsilon\leq\rho$, give $h,k\in\mathcal U_\rho$.

For item~\ref{item:product-pullback-forward}, put
$\xi_1=h+k\circ H=K\circ H-\Id$.
Apply Lemma~\ref{lem:product-Green-covariance} with
$\alpha=g$, $\beta=f$, $\psi=h$, and $u=k$.  Since $k$ is admissible
for $g$, this shows that $\xi_1$ is admissible for $f$ and gives
\[
 C_H\mathcal S_g(k)=\mathcal S_f(\xi_1)C_H.
\]
Indeed, the definitions of $H$, $K$, and $\xi_1$ give
$
 \Phi_k(H(x)) = \Phi_{\xi_1}(x).
$
Consequently, by \eqref{eq:product-field-coefficients}, the derivative
and cutoff maps after pullback satisfy
\begin{align*}
 A_k(H(x))
 &=Df(\Phi_k(H(x)))=Df(x+\xi_1(x))=A_{\xi_1}(x),\\
 E_k(H(x))
 &=E(\Phi_k(H(x)))=E(x+\xi_1(x))=E_{\xi_1}(x).
\end{align*}
These are exactly the coefficient maps defining
$\mathcal S_f(\xi_1)$.  For $O_f=\mathcal S_f(0)$, the corresponding
maps are
$
 A_0(x)=Df(x)$, $E_0(x)=E(x).
$
We therefore compare $\mathcal S_f(\xi_1)$ and $\mathcal S_f(0)$
using Lemma~\ref{lem:product-Salpha}\ref{item:product-Salpha-Lip} with
$\alpha=f$, $u=\xi_1$, and $\widetilde u=0$; the zero field is
admissible for $f$, and the lemma bounds the difference in
terms of $|\xi_1|_\infty$ without assuming $\xi_1=0$.
In particular, admissibility does not require $\xi_1\in\mathcal U_\rho$:
the corrected trajectories are
\[
 \Phi_{\xi_1}(f^r(x))=K(H(f^r(x)))=K(g^r(H(x))),\qquad r\in\Zz,
\]
which are already controlled by \eqref{eq:product-field-defect} with
$u=k$.  Although only $|\xi_1|_\infty\leq2\rho$ is known,
$2\rho<\rho_*$ allows us to apply \eqref{eq:product-Salpha-Lip}.
By \eqref{eq:product-Of},
\begin{equation}
\label{eq:product-pullback-series-difference}
 C_H\mathcal S_g(k)-O_fC_H
 =\bigl(\mathcal S_f(\xi_1)-\mathcal S_f(0)\bigr)C_H.
\end{equation}
Combining \eqref{eq:product-pullback-series-difference} with
\eqref{eq:product-Salpha-Lip} and $\|C_H\|\leq1$ yields
\begin{equation}
\label{eq:product-pullback-comparison}
 \|C_H\mathcal S_g(k)-O_fC_H\|\leq L_1|\xi_1|_\infty.
\end{equation}
The definitions of $h$ and $k$, together with
\eqref{eq:product-Og-definition}, now give
\[
 \xi_1
 =O_fw_f+C_H\mathcal S_g(k)w_g
 =O_f(w_f+C_Hw_g)
 +(C_H\mathcal S_g(k)-O_fC_H)w_g.
\]
Using \eqref{eq:product-Of-bound},
\eqref{eq:product-pullback-comparison}, and
$|w_g|_\infty\leq\varepsilon$, we obtain
\[
 |\xi_1|_\infty\leq L|w_f+w_g\circ H|_\infty
                    +L_1\varepsilon|\xi_1|_\infty.
\]
Since $L_g=L/(1-L_1\varepsilon)$, this proves
\eqref{eq:product-pullback-forward}.

For item~\ref{item:product-pullback-reverse}, apply
Lemma~\ref{lem:product-Green-covariance} with
$\alpha=f$, $\beta=g$, $\psi=k$, and $u=0$.
Here $v=K-\Id=k$, and \eqref{eq:product-Of} gives
the exact identity
\begin{equation}
\label{eq:product-pullback-operator}
 \mathcal S_g(k)C_K=C_KO_f.
\end{equation}
Under pullback by $K$, the corrected point for
$O_f=\mathcal S_f(0)$ becomes $K(y)=y+k(y)$.  The derivative and cutoff
are therefore evaluated at exactly the points used by $\mathcal S_g(k)$.
Consequently, $h\circ K$ and $k$ are represented by the same Green
operator, and their forcings can be added directly.
Put $\xi_2=k+h\circ K=H\circ K-\Id$.
By the normalization of $O_g$, the identity
\eqref{eq:product-pullback-operator}, and linearity in the forcing,
\[
 \xi_2=\mathcal S_g(k)w_g+C_KO_fw_f
     =\mathcal S_g(k)(w_g+C_Kw_f).
\]
The bound \eqref{eq:product-Salpha-bound} proves
\eqref{eq:product-pullback-reverse}.
\end{proof}

\begin{remark}
For sufficiently small fields and forcings, freezing the Green operator
at $u=0$ and using \eqref{eq:product-Df-Lip} gives a contraction
construction of solutions to \eqref{eq:product-Og-cohomological}.
This construction need not provide the normalization \eqref{eq:product-Og-definition}
and the covariance needed for the inverse identities.
\end{remark}

\subsection{The remainder operators}
\label{sec:product-remainders}

\begin{lem}
\label{lem:product-remainders}
\begin{enumerate}[label=\textup{(\roman*)},ref=(\roman*),leftmargin=10mm]
\item\label{item:product-remainders-bounds}
\textbf{Bounds and Lipschitz estimates.}
The operators $G$ and $R$ map $\mathcal U_\rho$ into
$\mathcal W_\varepsilon$.  They satisfy
\begin{equation}
 \label{eq:product-Gf}
 |G(h)|_\infty\leq d+\frac M2\rho^2=\varepsilon,
 \qquad
 |G(h)-G(\widetilde h)|_\infty \leq(M\rho+d)|h-\widetilde h|_\infty,
\end{equation}
\begin{equation}
\label{eq:product-R-field-estimates}
 |R(k)|_\infty\leq\varepsilon,
 \qquad
 |R(k)-R(\widetilde k)|_\infty
 \leq M\rho|k-\widetilde k|_\infty.
\end{equation}

\item\label{item:product-remainders-cancellation}
\textbf{Cancellation estimates.}
For arbitrary $h,k\in\mathcal U_\rho$, put
$H=\Id+h$ and $K=\Id+k$.  Then
\begin{equation}
\label{eq:product-remainder-KH}
 |G(h)+R(k)\circ H|_\infty
 \leq M\rho|h+k\circ H|_\infty,
\end{equation}
\begin{equation}
\label{eq:product-remainder-HK}
 |R(k)+G(h)\circ K|_\infty
 \leq(M\rho+d)|k+h\circ K|_\infty.
\end{equation}
These two estimates do not require either semiconjugacy identity.
\end{enumerate}
\end{lem}

\begin{proof}
For item~\ref{item:product-remainders-bounds}, the defining formulas give
continuous fields for inputs in $\mathcal U_\rho$.  The elementary estimates
\eqref{eq:product-s-Lip-v} and \eqref{eq:product-Delta-estimates}
give \eqref{eq:product-Gf}.
To prove \eqref{eq:product-R-field-estimates}, we estimate
$s_f(y+v,-v)$ directly.  For a scalar base point $t$, differentiation
gives
\begin{equation*}
 \frac{d}{dq}\bigl(a(t)-a(t+q)+a'(t+q)q\bigr)
 =a''(t+q)q.
\end{equation*}
Since $|a''|\leq M$, this derivative has absolute value at most
$M\rho$ for $|q|\leq\rho$.  Applying the resulting Lipschitz estimate
coordinatewise and taking the $\ell^p$ norm gives
\begin{equation}
\label{eq:product-endpoint-remainder-Lip}
\begin{gathered}
 |s_f(y+v,-v)-s_f(y+\widetilde v,-\widetilde v)|_p
 \leq M\rho|v-\widetilde v|_p,\\
 |v|_p,|\widetilde v|_p\leq\rho.
\end{gathered}
\end{equation}
For $k\in\mathcal U_\rho$, the defining formulas give
\begin{equation}\label{eq:Rk-F}
 R(k)(y)=-s_f(y+k(y),-k(y))-\Delta(y).
\end{equation}
The Taylor bound $|s_f(y+v,-v)|_p\leq M|v|_p^2/2$ and
\eqref{eq:product-Delta-estimates} therefore imply
\[
 |R(k)(y)|_p
 \leq \frac M2|k(y)|_p^2+d
 \leq \frac M2\rho^2+d=\varepsilon.
\]
Applying \eqref{eq:product-endpoint-remainder-Lip} with $v=k(y)$
and $\widetilde v=\widetilde k(y)$ also gives
\[
 |R(k)(y)-R(\widetilde k)(y)|_p
 \leq M\rho|k(y)-\widetilde k(y)|_p.
\]
Taking the supremum over $y$ in both bounds proves
\eqref{eq:product-R-field-estimates}.

For item~\ref{item:product-remainders-cancellation}, we first record the
exact cancellation underlying \eqref{eq:product-remainder-KH}.
If $k\in\mathcal U_\rho$ and the right-hand side of
\eqref{eq:product-remainder-KH} vanishes, that is, $k(H(x))=-h(x)$ for
all $x\in\Bb$, then
\[
 R(k)(x+h(x))
 =f(x)-f(x+h(x))+Df(x)h(x)-\Delta(x+h(x))
 =-G(h)(x),
\]
and hence the left-hand side vanishes as well.

For arbitrary $k$, the definitions of $G$ and $R$ give
\begin{equation*}
 G(h)(x)+R(k)(H(x))
 =s_f(x,h(x))-s_f(K(H(x)),-k(H(x))).
\end{equation*}
Apply \eqref{eq:product-endpoint-remainder-Lip} with
$y=H(x)$, $v=-h(x)$, and $\widetilde v=k(H(x))$.
Both displacements have norm at most $\rho$, and
$y+v=x$, $y+\widetilde v=K(H(x))$.  Hence
\[
 |G(h)(x)+R(k)(H(x))|_p
 \leq M\rho|h(x)+k(H(x))|_p.
\]
Taking the supremum over $x$ proves
\eqref{eq:product-remainder-KH}.
Similarly, \eqref{eq:Rk-F} implies
\begin{align*}
 R(k)(y)+G(h)(K(y))
 =s_f(K(y),h(K(y)))-s_f(K(y),-k(y))\notag
 +\Delta(H(K(y)))-\Delta(y)
\end{align*}
Applying \eqref{eq:product-s-Lip-v} to the first difference and the
Lipschitz bound on $\Delta$ to the second, and using
$H(K(y))-y=k(y)+h(K(y))$, we obtain
\eqref{eq:product-remainder-HK}.
\end{proof}

\subsection{Structural stability}
\label{sec:product-structural-stability}

\begin{proof}[Proof of Theorem~\ref{thm:product-structural-stability}]
For $\Delta=g-f$, assumption \eqref{eq:product-fg-close} implies
\eqref{eq:product-Delta-estimates}.  Take the constants
$\delta_*,L,L_1,\rho_*$ as above and set
$\rho=2Ld$.
We choose $d_0>0$ so small that, whenever $0<d<d_0$, all inequalities
in \eqref{eq:product-smallness-list} hold.
This is possible because $\rho=2Ld\to0$ as $d\to0$.

\medskip
\noindent\textbf{Step 1: the forward semiconjugacy.}
Use the closed ball $\mathcal U_\rho$ from
\eqref{eq:product-forward-ball-space} and the operator $G$ from
\eqref{eq:product-G-def}.
By Lemmas~\ref{lem:product-Of} and
\ref{lem:product-remainders}~\ref{item:product-remainders-bounds},
\begin{align*}
 |O_fG(h)|_\infty
 &\leq L\left(d+\frac M2\rho^2\right)\leq\rho,\\
 |O_fG(h)-O_fG(\widetilde h)|_\infty
 &\leq L(M\rho+d)|h-\widetilde h|_\infty.
\end{align*}
Since $L(M\rho+d)<1/2$ by \eqref{eq:product-smallness-list},
$h\mapsto O_fG(h)$ is a contraction of $\mathcal U_\rho$ into itself.
Its unique fixed point $h$ satisfies
$\mathcal D_fh=G(h)$ by \eqref{eq:product-Of-bound}.
The definition of $G$ and \eqref{eq:nonlinear-cohom} then give
$g\circ H=H\circ f$ for $H=\Id+h$.

\medskip
\noindent\textbf{Step 2: a normalized reverse semiconjugacy.}
Lemmas~\ref{lem:product-Og} and
\ref{lem:product-remainders}~\ref{item:product-remainders-bounds} show that
$k\mapsto O_gR(k)$ is well defined on $\mathcal U_\rho$ and satisfies
\begin{equation*}
 |O_gR(k)|_\infty\leq L\varepsilon\leq\rho, \qquad
 |O_gR(k)-O_gR(\widetilde k)|_\infty
 \leq L_gM\rho|k-\widetilde k|_\infty.
\end{equation*}
Hence this map has a unique fixed point $k\in\mathcal U_\rho$.
Define
$K:=\Id+k$.
Equation \eqref{eq:product-Og-cohomological} with $w=R(k)$,
together with the definition of $R$ and $g=f+\Delta$, gives
$K\circ g=f\circ K$.

\medskip
\noindent\textbf{Step 3: the semiconjugacies are mutually inverse.}
Put $w_f=G(h)$ and $w_g=R(k)$; these forcings belong to
$\mathcal W_\varepsilon$ by
Lemma~\ref{lem:product-remainders}~\ref{item:product-remainders-bounds}.
Since $g\circ H=H\circ f$ and $K\circ g=f\circ K$, applying
Lemma~\ref{lem:product-pullback}~\ref{item:product-pullback-forward}
and~\ref{item:product-pullback-reverse}, together with
Lemma~\ref{lem:product-remainders}~\ref{item:product-remainders-cancellation},
gives
\begin{align*}
 |h+k\circ H|_\infty
 &\leq L_g|G(h)+R(k)\circ H|_\infty
 \leq L_gM\rho|h+k\circ H|_\infty,\\
 |k+h\circ K|_\infty
 &\leq L|R(k)+G(h)\circ K|_\infty
 \leq L(M\rho+d)|k+h\circ K|_\infty.
\end{align*}
Here $L_gM\rho<1$  and
$L(M\rho+d)<1/2$ by \eqref{eq:product-smallness-list}.
Thus both composition defects vanish, so
$K\circ H=\Id$ and
$H\circ K=\Id$.
Both $H$ and $K$ are continuous, so they are mutually inverse
homeomorphisms.  Finally,
$|h|_\infty,|k|_\infty\leq\rho=2Ld$, and the theorem holds with
$
 M_0=2L.
$
\end{proof}

\begin{remark}
\label{rem:product-uniform-conjugacy}
Let $BUC(\Bb,\Bb)$ denote the closed subspace of $C_b^0(\Bb,\Bb)$
consisting of bounded uniformly continuous fields.  The corrections in
Theorem~\ref{thm:product-structural-stability} can be chosen with
$h,k\in BUC(\Bb,\Bb)$, with the same bounds.  Thus both $H$ and
$H^{-1}$ are uniformly continuous.
Indeed, the maps $f^{\pm1}$ and $g^{\pm1}$ are Lipschitz; in particular,
the argument in Subsection~\ref{sec:product-cohomological-operators}
gives $\Lip(g^{-1})\leq(\lambda_1-d)^{-1}$.
The coefficient maps $Df$, $x\mapsto Df(x)^{-1}$, and $E$ are bounded
and Lipschitz.  Hence, if $u,w\in BUC(\Bb,\Bb)$ and $u$ is admissible
for $\alpha\in\{f,g\}$, every finite partial sum of
$\mathcal S_\alpha(u)w$ is uniformly continuous: it uses finitely many
Lipschitz iterates of $\alpha^{\pm1}$ and the uniformly continuous map
$\Id+u$.  Uniform convergence of the Green series gives
$\mathcal S_\alpha(u)w\in BUC(\Bb,\Bb)$.
The operators $G$ and $R$ also preserve bounded uniform continuity
on $\mathcal U_\rho$.  Starting from zero, the iterations defining
$h=O_fG(h)$, $O_gw$, and $k=O_gR(k)$ therefore remain in this closed
subspace and converge uniformly, which gives the asserted regularity
with the same bounds.

An analogous upgrade in the general part of the paper requires
additional reasoning. Continuous selections, whether from Michael's theorem or from a partition of unity, come without a uniform modulus of continuity, so the subsequent
uniform-limit arguments do not by themselves yield semiconjugacies in
$\Id+BUC(\Bb,\Bb)$.
\end{remark}

\begin{remark}
\label{rem:product-regularity}
The $C^2$ assumption on $a$ enters the proof only through the bound
$|a''|\leq M$ in \eqref{eq:product-a-bounds}.  A reduction to
$a\in C^1(\Rr)$ with bounded uniformly continuous derivative would
require additional arguments, since the contractions and pullback
estimates above use Lipschitz dependence on the displacement.
\end{remark}

\appendix
\section{Proofs of the scalar estimates}
\label{app:product-scalar-pseudotrajectories}

\subsection{Cutoff and trapping}

\begin{proof}[Proof of Lemma~\ref{lem:product-scalar-trapping}]
\medskip
\noindent\textbf{Step 1: construction of the scalar cutoff.}
Since $a^{-1}(-c)>-c$ and $a^{-1}(c)<c$, we can choose $\gamma>0$ so
small that
\begin{equation*}
 -c+4\gamma<a^{-1}(-c-4\gamma)<0
 <a^{-1}(c+4\gamma)<c-4\gamma.
\end{equation*}
Choose a Lipschitz function $\eta:\mathbb R\to[0,1]$ satisfying
\begin{equation*}
 \begin{aligned}
 \eta(t)&=1
 &&\text{for }t\leq-c+4\gamma
       \text{ or }t\geq c-4\gamma,\\
 \eta(t)&=0
 &&\text{for }a^{-1}(-c-4\gamma)
       \leq t\leq a^{-1}(c+4\gamma).
 \end{aligned}
\end{equation*}
On each of the two remaining compact intervals, we use linear
interpolation or any other Lipschitz interpolation.  These properties give
\eqref{eq:product-eta-implications}, proving item~\ref{item:product-scalar-cutoff}.

\medskip
\noindent\textbf{Step 2: scalar pseudotrajectories remain in the correct
regions.}
Put
\begin{equation}
\label{eq:product-R}
 \Lambda:=\lambda_1^{-1}.
\end{equation}
By the monotonicity of $a$ and \eqref{eq:product-a-order}, the numbers
\begin{equation*}
 \alpha_s:=\min\left\{
 a(c+2\gamma)-(c+2\gamma),
 (-c-2\gamma)-a(-c-2\gamma)
 \right\}
\end{equation*}
and
\begin{equation*}
 \alpha_u:=\min\left\{
 (c-2\gamma)-a^{-1}(c-2\gamma),
 a^{-1}(-c+2\gamma)-(-c+2\gamma)
 \right\}
\end{equation*}
are positive.  We first choose $\delta_*>0$ so small that
\begin{equation}
\label{eq:product-delta-margin}
 \delta_*<\min\{\gamma,\alpha_s/2,(2\Lambda)^{-1}\alpha_u\}.
\end{equation}
Further decreases of $\delta_*$ preserve the conclusions below.

Let $(t_n)_{n\in\Zz}$ satisfy \eqref{eq:product-scalar-pseudo}.
If $\eta(t_r)\neq0$, then \eqref{eq:product-eta-implications} and
\eqref{eq:product-scalar-pseudo} imply
$|t_{r+1}|>c+3\gamma$.  Once $|t_k|\geq c+2\gamma$, monotonicity of $a$
and the definition of $\alpha_s$ imply that $a(t_k)$ lies at least
$\alpha_s$ beyond the boundary of the set
$(-\infty,-c-2\gamma]\cup[c+2\gamma,+\infty)$.  Put
$\beta:=\min\{\gamma,\alpha_s/2,\alpha_u/2\}>0$.
Then \eqref{eq:product-delta-margin} and induction give
\eqref{eq:product-stable-margin}.

For the backward direction, \eqref{eq:product-a-bounds} gives
$\operatorname{Lip}(a^{-1})\leq \Lambda$.  Therefore
\begin{equation}
\label{eq:product-backward-error}
 |t_k-a^{-1}(t_{k+1})|
 \leq \Lambda|a(t_k)-t_{k+1}|
 <\Lambda\delta_*.
\end{equation}
If $1-\eta(t_r)\neq0$, then $|t_r|<c-4\gamma$.
The definition of $\alpha_u$, \eqref{eq:product-delta-margin}, and
backward induction give \eqref{eq:product-unstable-margin}.
This proves item~\ref{item:product-scalar-margins}.
\end{proof}

\subsection{Uniform exponential bounds}

\begin{proof}[Proof of Lemma~\ref{lem:product-scalar-pseudotrajectories}]
Keep the cutoff and margins from Lemma~\ref{lem:product-scalar-trapping},
and recall that $\Lambda=\lambda_1^{-1}$ from \eqref{eq:product-R}.
We prove the subsegment estimates by comparing finite pseudotrajectory
blocks with exact orbits and then concatenating blocks.

\medskip
\noindent\textbf{Step 1: forward products.}
Choose $N\geq1$ so large that
\begin{equation}
\label{eq:product-N}
 C\lambda^N<\frac14.
\end{equation}
Consider a scalar pseudotrajectory segment
$t_0,t_1,\ldots,t_N$ contained in the stable region
$|t_i|>c+2\gamma$.  Let $u_i=a^i(t_0)$.  From
$|a'|\leq \Lambda$ and \eqref{eq:product-scalar-pseudo}, induction gives
\begin{equation}
\label{eq:product-finite-comparison}
 |t_i-u_i|
 \leq\delta_*\sum_{q=0}^{i-1}\Lambda^q,
 \qquad 0\leq i\leq N.
\end{equation}
For arbitrary numbers $\alpha_i,\beta_i$, one has the product telescoping
identity
\begin{equation}
\label{eq:product-product-telescope}
 \prod_{i=0}^{N-1}\alpha_i-
 \prod_{i=0}^{N-1}\beta_i
 =\sum_{q=0}^{N-1}
 \left(\prod_{i<q}\alpha_i\right)
 (\alpha_q-\beta_q)
 \left(\prod_{i>q}\beta_i\right).
\end{equation}
Apply this with $\alpha_i=a'(t_i)$ and $\beta_i=a'(u_i)$.  Using
\eqref{eq:product-a-bounds}, \eqref{eq:product-finite-comparison}, and
$|a''|\leq M$, we obtain
\begin{equation*}
 \left|
 \prod_{i=0}^{N-1}a'(t_i)-
 \prod_{i=0}^{N-1}a'(u_i)
 \right|
 \leq
 M \Lambda^{N-1}N\left(\sum_{q=0}^{N-1}\Lambda^q\right)\delta_*.
\end{equation*}
After decreasing $\delta_*$, the right-hand side is less than $1/4$.
Because $|t_0|>c$, \eqref{eq:product-scalar-estimates} and
\eqref{eq:product-N} imply
$
 \prod_{i=0}^{N-1}a'(u_i)=|(a^N)'(t_0)|<\frac14.
$
Thus every stable pseudotrajectory block of length $N$ has a derivative
product of at most $1/2$.

\medskip
\noindent\textbf{Step 2: backward products.}
Use the same argument with $a^{-1}$ and the factors
 $b(t)=\frac1{a'(t)}$.
By \eqref{eq:product-a-bounds},
\begin{equation}
\label{eq:product-b-bounds}
 |b(t)|\leq \Lambda,
 \qquad
 |b'(t)|=\frac{|a''(t)|}{|a'(t)|^2}
 \leq M\lambda_1^{-2}.
\end{equation}
For a backward pseudotrajectory block $t_{r-N},\ldots,t_r$ contained
in the region $|t|<c-2\gamma$, put $u_{r-q}=a^{-q}(t_r)$, $0\leq q\leq N$.
By \eqref{eq:product-backward-error}, the one-step error is bounded
by $\Lambda\delta_*$, so
\begin{equation*}
 |t_{r-q}-u_{r-q}|
 \leq \Lambda\delta_*\sum_{i=0}^{q-1}\Lambda^i,
 \qquad 1\leq q\leq N.
\end{equation*}
The product telescoping identity \eqref{eq:product-product-telescope}
and \eqref{eq:product-b-bounds} therefore give
\begin{equation*}
 \left|\prod_{k=r-N}^{r-1}b(t_k)
 -\prod_{k=r-N}^{r-1}b(u_k)\right|
 \leq C_N\delta_*,
\end{equation*}
where $C_N$ depends only on the scalar bounds and the chosen block
length $N$.  The exact product satisfies
$ \prod_{k=r-N}^{r-1}b(u_k)
 =|(a^{-N})'(t_r)|
 <\frac14
$
by \eqref{eq:product-scalar-estimates} and \eqref{eq:product-N}.
Decreasing $\delta_*$ so that $C_N\delta_*<1/4$ gives the same
bound $1/2$ for every backward block of length $N$.

\medskip
\noindent\textbf{Step 3: concatenation of blocks.}
Split either type of segment of length $m$ into blocks of length $N$
and a remainder of length less than $N$.  Each full-block product is
at most $1/2$, and each factor in the remainder is at most $\Lambda$.
Thus each product, in its respective region, is bounded by
\[
 \Lambda^{N-1}2^{-\lfloor m/N\rfloor}
 \leq 2\Lambda^{N-1}\bigl(2^{-1/N}\bigr)^m.
\]
We may therefore take
$\theta=2^{-1/N}$ and $C_0=2\Lambda^{N-1}$.
This gives \eqref{eq:product-stable-subsegment} and
\eqref{eq:product-unstable-subsegment}, completing the proof.
\end{proof}

\subsection{Green coefficient bounds and Lipschitz dependence}

\begin{proof}[Proof of Lemma~\ref{lem:product-scalar-Green-Lipschitz}]
\medskip
\noindent\textbf{Step 1: coefficient bounds.}
Let $\widehat C_0$ be the subsegment constant supplied by
Lemma~\ref{lem:product-scalar-pseudotrajectories} before any further
increase, and put $\Lambda=\lambda_1^{-1}$.  If $\eta(t_r)\neq0$, the forward
margin in Lemma~\ref{lem:product-scalar-trapping} and
\eqref{eq:product-stable-subsegment} give
\[
 |b^s_{n,r}(t)|\leq\widehat C_0\theta^{n-r-1},\qquad r<n.
\]
If $1-\eta(t_r)\neq0$, the backward margin and
\eqref{eq:product-unstable-subsegment} control the product from $n$ to
$r-1$.  The last factor $1/a'(t_r)$ is at most $\Lambda$, so
\[
 |b^u_{n,r}(t)|\leq \Lambda\widehat C_0\theta^{r-n},\qquad n\leq r.
\]
When the corresponding cutoff vanishes, the coefficient is zero.
Thus replacing $C_0$ by
$\max\{\widehat C_0,\Lambda\widehat C_0/\theta\}$ proves
\eqref{eq:product-stable-coefficient} and
\eqref{eq:product-unstable-coefficient}, including the empty-product
cases.  The subsegment estimates remain valid with this larger constant.

\medskip
\noindent\textbf{Step 2: Lipschitz dependence.}
Choose $\rho_*\in(0, \beta/2)$.
Let $t,\widetilde t$ satisfy the assumptions of the lemma, and put
$
 \varepsilon:=\sup_{k\in\Zz}|t_k-\widetilde t_k|<\rho_*$,
 $L_\eta:=\operatorname{Lip}(\eta).
$
We first check that the subsegment estimates apply to both sequences
whenever at least one of the corresponding cutoff factors is nonzero.
If $\eta(t_r)\neq0$, then \eqref{eq:product-stable-margin} gives, for
$k\geq r+1$,
\[
 |\widetilde t_k|\geq |t_k|-\varepsilon
 >c+2\gamma+\beta-\varepsilon>c+2\gamma.
\]
The same argument with $t,\widetilde t$ interchanged applies if
$\eta(\widetilde t_r)\neq0$.  Similarly, if $1-\eta(t_r)\neq0$, then
\eqref{eq:product-unstable-margin} gives, for $k\leq r$,
\[
 |\widetilde t_k|\leq |t_k|+\varepsilon
 <c-2\gamma-\beta+\varepsilon<c-2\gamma,
\]
and again the roles of the two sequences can be interchanged.
Thus, in each nontrivial case, both sequences lie in the region required
by the corresponding subsegment estimate.

For integers $i\leq j$, abbreviate
\[
 \Pi^s_{i,j}(t):=\prod_{k=i}^{j-1}a'(t_k),
 \qquad
 \Pi^u_{i,j}(t):=\prod_{k=i}^{j-1}b(t_k),
 \qquad b(t)=\frac1{a'(t)}.
\]
Empty products equal $1$.  By
\eqref{eq:product-stable-subsegment} and
\eqref{eq:product-unstable-subsegment}, the relevant products along
either sequence are bounded by $C_0\theta^{j-i}$, including the case $i=j$,
since $C_0\geq1$.

\medskip
\noindent\emph{Stable coefficients.}
Fix $r<n$ and put $m=n-r-1$.  If
$\eta(t_r)=\eta(\widetilde t_r)=0$, the assertion is immediate.
Otherwise, both sequences lie in the stable region for $k\geq r+1$,
and
$$
 b^s_{n,r}(t)-b^s_{n,r}(\widetilde t)
 =\bigl(\eta(t_r)-\eta(\widetilde t_r)\bigr)
       \Pi^s_{r+1,n}(\widetilde t)
 +\eta(t_r)
       \bigl(\Pi^s_{r+1,n}(t)-\Pi^s_{r+1,n}(\widetilde t)\bigr).
$$
The first term has absolute value at most
$L_\eta C_0\theta^m\varepsilon$.
For $m\geq1$, the product telescoping identity
\eqref{eq:product-product-telescope} reads
\[
 \Pi^s_{r+1,n}(t)-\Pi^s_{r+1,n}(\widetilde t)
 =\sum_{j=r+1}^{n-1}
 \Pi^s_{r+1,j}(t)
 \bigl(a'(t_j)-a'(\widetilde t_j)\bigr)
 \Pi^s_{j+1,n}(\widetilde t).
\]
For each summand,
$$
 |\Pi^s_{r+1,j}(t)|\leq C_0\theta^{j-r-1}, \quad
 |a'(t_j)-a'(\widetilde t_j)|\leq M\varepsilon, \quad
 |\Pi^s_{j+1,n}(\widetilde t)|\leq C_0\theta^{n-j-1}.
 $$
The two products are estimated separately along their respective
pseudotrajectories.  Their lengths add up to $m-1$, so each of the
$m$ summands is bounded by $MC_0^2\theta^{m-1}\varepsilon$.
Using $0\leq\eta\leq1$, we obtain
\begin{equation*}
 |b^s_{n,r}(t)-b^s_{n,r}(\widetilde t)|
 \leq
 \bigl(L_\eta C_0\theta^m
       +mMC_0^2\theta^{\max\{m-1,0\}}\bigr)\varepsilon.
\end{equation*}
This also holds for $m=0$: the products are empty, and the difference
is simply $\eta(t_r)-\eta(\widetilde t_r)$.

\medskip
\noindent\emph{Unstable coefficients.}
The argument is analogous, with $a'$ replaced by $b=1/a'$.
For $m=r-n\geq0$, separate the final factor:
$
 b^u_{n,r}(t)=(1-\eta(t_r))b(t_r)\Pi^u_{n,r}(t).
$
In the nontrivial case, the remaining product uses only the controlled
segment from $n$ to $r$.
By \eqref{eq:product-b-bounds}, $|b|\leq\Lambda$ and
$\operatorname{Lip}(b)\leq M\Lambda^2$, so the endpoint factor satisfies
\[
 |(1-\eta)b|_\infty\leq\Lambda,\qquad
 \operatorname{Lip}((1-\eta)b)\leq\Lambda L_\eta+M\Lambda^2.
\]
The same cutoff decomposition and telescoping estimate therefore give
$$
 |b^u_{n,r}(t)-b^u_{n,r}(\widetilde t)|
 \leq\bigl((\Lambda L_\eta+M\Lambda^2)C_0\theta^m
          +mM\Lambda^3C_0^2
             \theta^{\max\{m-1,0\}}\bigr)\varepsilon.
$$
As in the stable case, this includes $m=0$, when the product is empty.

Replacing $\theta$ by any slightly larger $\theta_1\in(\theta,1)$
absorbs the polynomial factors into a constant.
After we also absorb the fixed factor $\theta_1^{-1}$ in the unstable
case, \eqref{eq:product-scalar-Green-Lip-s} and
\eqref{eq:product-scalar-Green-Lip-u} follow with a common constant $C_4$.
All constants depend only on $a$.
\end{proof}

\section*{Declarations}

\paragraph{Acknoledgement.} 
The author thanks S. Yu. Pilyugin for stimulating discussions.

\paragraph{Funding.}
The work was supported by FAPERJ APQ1 E-26/210.702/2024 (295291),
FAPESP grant 2023/07076-4, and the National Council for Scientific and
Technological Development (CNPq), through grants 404123/2023-6 and
403673/2025-9 and a Bolsa de Produtividade em Pesquisa do CNPq -- Nível C.

\paragraph{Competing interests.}
The author has no competing interests to declare that are relevant to
the content of this article.

\paragraph{Data availability.}
Data sharing is not applicable to this article, as no datasets were
generated or analysed during the current study.

\paragraph{Use of generative AI.}
During the preparation of this work, the author used ChatGPT (OpenAI)
and Claude (Anthropic).
The author provided a detailed draft of the proof of the first
semiconjugacy in Theorem \ref{thm:main}, which ChatGPT expanded into a
complete text. The author guided the construction of the reverse
semiconjugacy step by step through detailed prompts. For the second main
theorem, the author provided the strategy and the main technical tools:
the cutoff map $\eta$, the operator $\mathcal S_{\alpha}(u)$, the covariance identity
in Lemma \ref{lem:product-Green-covariance}, the nonlinear cohomological
equation, pullback compatibility in Lemma \ref{lem:product-pullback}, and
the cancellation of the remainders in
Lemma~\ref{lem:product-remainders}~\ref{item:product-remainders-cancellation}.
The wording was prepared with the help of ChatGPT under the author's
direction.
Claude was used to review the manuscript and to suggest and draft, at the
author's request, revisions of the exposition, structure and notation.
The author reviewed and edited all suggestions and takes full responsibility for the content of the article.

Although most of the text was prepared before 11 September 2026, the
author's use of AI followed the principles described in the blog post
cited below\footnote{\emph{A Severe Misalignment of AI in Mathematics},
\url{https://terrytao.wordpress.com/2026/09/11/a-severe-misalignment-of-ai-in-mathematics/}},
including its use to develop new concepts and understand mathematical
phenomena.

\end{document}